\documentclass[a4paper,10pt]{article}
\title{Random sections of line bundles over real Riemann surfaces}
\author{Michele Ancona \thanks{Institut Camille Jordan, Umr Cnrs 5208, Universit\'{e} Claude Bernard Lyon 1. ancona@math.univ-lyon1.fr} }
\date{}
\usepackage[a4paper,bindingoffset=0.2in,%
            left=1in,right=1in,top=1in,bottom=1in,%
            footskip=.25in]{geometry}
\usepackage[latin1]{inputenc}
\usepackage[T1]{fontenc}
\usepackage[english]{babel}
\usepackage{amsmath,amssymb,amsthm,amscd}
\usepackage{indentfirst}
\usepackage{fancyhdr}

\usepackage{young}
\usepackage{graphicx}

\usepackage[vcentermath]{youngtab}

\RequirePackage[colorlinks,linkcolor=blue,citecolor=blue,urlcolor=blue]{hyperref}
\usepackage{enumerate}
\usepackage{color}
\usepackage{pst-fill,pst-grad,pst-plot,pst-eucl,pstricks-add,pst-node}

\usepackage{dsfont}

\theoremstyle{plain}
\newtheorem{thm}{Theorem}[section]

\newtheorem{prop}[thm]{Proposition}
\newtheorem{lem}[thm]{Lemma}
\newtheorem{cor}[thm]{Corollary}
\newtheorem{oss}[thm]{Remark}

\theoremstyle{definition}
\newtheorem{defn}[thm]{Definition}
\newtheorem{conv}[thm]{Notation}

\newcommand{\R}{\mathbb{R}}
\newcommand{\C}{\mathbb{C}}
\newcommand{\Jac}{\textrm{Jac}}

\newcommand{\Vol}{\textrm{Vol}}

\newcommand{\dV}{\textrm{dV}}
\newcommand{\Fix}{\textrm{Fix}}
\newcommand{\norm}[1]{\left\lVert #1 \right\rVert}

\allowdisplaybreaks

\begin{document}
\maketitle
\begin{abstract}
Let   $\mathcal{L}$ be a positive line bundle over a Riemann surface $\Sigma$ defined over $\mathbb{R}$. We prove that sections $s$ of $\mathcal{L}^d$, $d\gg 0$, whose number of real zeros $\#Z_s$ deviates from the expected one are rare. We also provide  asymptotics 
 of the form $\mathbb{E}[(\#Z_s-\mathbb{E}[\# Z_s])^k]=O(\sqrt{d}^{k-1-\alpha})$ and ${\mathbb{E}[\#Z^k_s]=a_k\sqrt{d}^{k}+b_k\sqrt{d}^{k-1}+O(\sqrt{d}^{k-1-\alpha})}$
 for all the (central) moments of the number of real zeros.
Here, $\alpha$ is  any number in $(0,1)$, and $a_k$ and $b_k$ are some explicit and positive constants.
Finally, we obtain similar asymptotics for the distribution of complex zeros of random sections.
Our proof involves Bergman kernel estimates as well as Olver multispaces.
\end{abstract}

\tableofcontents

\maketitle


\section{Introduction}\label{sec:intro}

Let $p\in\R_d[X]$ be a degree $d$ polynomial with real coefficients.
$$\emph{How many real roots does $p$ have, choosing it at random?}$$
A first answer  was given by Kac in the 40's. He proved that the expected number of real roots  $\mathbb{E}[\#Z_p]$ is equivalent to $\frac{2}{\pi}\log d$ as $d$ grows to $+\infty$, see \cite{kac}. 
By definition, 
$$\mathbb{E}[\#Z_p]=\int_{p\in\R_d[X]}(\#Z_p)d\mu(p)$$
where $\#Z_p$ denotes the cardinality of $Z_p=\{x\in\R\mid p(x)=0\}.$
The probability measure on $\R_d[X]$  Kac has considered was the Gaussian measure associated with a scalar product that makes $\{X^k\}_{0\leqslant k\leqslant d}$ an orthonormal basis.

In the 90's, another answer was given by  Kostlan  and by  Shub and  Smale. 
They proved that $\mathbb{E}[\#Z_p]=\sqrt{d}$ for any $d\in\mathbb{N}$, see \cite{ko,ss}.
There, the probability measure  on $\R_d[X]$ was the Gaussian measure  associated with a scalar product that turns $\{\sqrt{\binom{d}{k}}X^k\}_{0\leqslant k\leqslant d}$ into an orthonormal basis. This scalar product makes $\R_d[X]$ isometric to the space $\R H^0(\C P^1;\mathcal{O}(d))$ of real global sections of the line bundle $\mathcal{O}(d)$ over $\C P^1$, where the scalar product  on $\R H^0(\C P^1;\mathcal{O}(d))$ is the $L^2$-scalar product induced by the Fubini-Study metric on $\mathcal{O}(d)$, see Section \ref{generality1}.
This point of view brings us to consider a general Riemann surface $\Sigma$  with any ample line bundle over $\Sigma$, see \cite{gwexp,gw2,sz}.

\medskip

Let $\Sigma$ be a closed real  Riemann surface, that is a closed Riemann surface equipped with an anti-holomorphic involution $c_{\Sigma}$. We denote by
$\mathbb{R}\Sigma=\Fix(c_{\Sigma})$ its real locus.
Let $\mathcal{L}$  be a real ample line bundle over $\Sigma$, that is an ample holomorphic line bundle $p:\mathcal{L}\rightarrow \Sigma$  equipped with an anti-holomorphic involution $c_{\mathcal{L}}$ such that $p\circ c_{\mathcal{L}}=c_{\Sigma}\circ p$ and $c_{\mathcal{L}}$ is complex-antilinear in the fibers. A real Hermitian metric on $\mathcal{L}$ is a Hermitian metric $h$ such that $c_{\mathcal{L}}^*h=\bar{h}$.\\ 
A real Hermitian metric on $\mathcal{L}$ induces a $L^2-$scalar product, and thus a Gaussian measure, on the space of global holomorphic real sections of $\mathcal{L}^d$, see Section \ref{generality1}. We  denote such space by $\R H^0(\Sigma;\mathcal{L}^d)\doteqdot\{s\in H^0(\Sigma;\mathcal{L})\mid c_L\circ s=s\circ c_{\Sigma} \}.$  
 $$\emph{How many real zeros does a  random real section $s\in \R H^0(\Sigma;\mathcal{L}^d)$ have?} $$ 
 We denote by $Z_s=\{s=0\}\cap\R\Sigma$ the set of real zeros of $s\in\R H^0(\Sigma;\mathcal{L}^d)$.
In \cite{gw2} it is shown that $\displaystyle\lim_{d\rightarrow\infty}\frac{1}{\sqrt{d}}\mathbb{E}[\#Z_s]=\frac{\Vol_h(\R\Sigma)}{\sqrt{\pi}}$. We get
\begin{thm}\label{expon} Let $c(d)\in \R^*_+$ be any sequence of positive real numbers, then, for any $k\in\mathbb{N}$ and   any $\alpha\in (0,1)$, the following holds as $d\rightarrow\infty$:
\begin{itemize}
\item $\mu\big\{\big|\# Z_s-\mathbb{E}[\#Z_s]\big|>c(d)\sqrt{d}\big\}= O\big(\frac{1}{c(d)^k\sqrt{d}^{1+\alpha}}\big);$
\item for any measurable subset $A\subset \R\Sigma$ of positive volume: $$\mu\big\{\big|\# (Z_s\cap A)-\mathbb{E}[\#(Z_s\cap A)]\big|>c(d)\sqrt{d}\big\}= O\big(\frac{\log d}{c(d)^k\sqrt{d}}\big)\hspace{1.5mm}.$$
\end{itemize}
\end{thm}
The measure of the set of real sections whose number of  real zeros deviates from the expected one goes to zero faster than all polynomials. Moreover, this property is valid for any open subset of $\R\Sigma$.\\
\begin{oss} For  a real analytic metric $h$, it seems possible to improve the polynomial upper estimate of Theorem \ref{expon} by an exponential one. This is a work in progress. Recall that real analytic metrics are dense in the space of smooth  metrics (see \cite{bouche,tian}) and that the Fubini-Study metric on $\C P^1$ is real analytic. Remark that an exponential rarefaction for real sections with "many" real zeros is proved in \cite[Theorem 2]{gwexp}
\end{oss}

\paragraph*{Higher moments} The main ingredient of the proof of Theorem \ref{expon}  is an estimate of   all the  moments $\mathbb{E}[\#Z^k_s]$, $k\in\mathbb{N}$, which has its own interest.  By definition, the $k$-th moment  of $\# Z_s$ equals 
$$\displaystyle\mathbb{E}[\#Z^k_s]=\int_{s\in\R H^0(\Sigma;\mathcal{L}^d)}(\#Z_s)^kd\mu(s).$$
The computation  of $\mathbb{E}[\#Z^k_s]$ is actually the core of the paper.
\begin{thm}\label{asymom} There exists a universal positive constant $M$ such that, for any  positive real Hermitian line bundle $(\mathcal{L},h)$  over any real Riemann surface $\Sigma$, any integer $k$ and any $\alpha\in(0,1)$,  the following asymptotics hold:
 $$\frac{1}{\sqrt{d}^k}\mathbb{E}[\#Z_s^k]=\frac{\Vol_h(\R \Sigma)^k}{\sqrt{\pi}^k}+\frac{Mk(k-1)\Vol_h(\R\Sigma)^{k-1}}{2\sqrt{\pi}^{k-2}\sqrt{d}}+O\big(\frac{1}{\sqrt{d}^{1+\alpha}}\big).$$ 
   \end{thm}
By universal constant we mean that it is a constant which neither depends on $\Sigma$ nor on $k\in\mathbb{N}$.
This formula was  known for $k=1$ and $k=2$ (see resp. \cite[Theorem 1.1]{gw2} or \cite[Theorem 1.3]{let} and \cite[Theorem 1.6]{Puchol}) and  unknown in this general case.
Theorem \ref{asymom} is a consequence of a more precise equidistribution result, namely Theorem \ref{equimom} below. Before stating this, we need some definitions.
\begin{itemize}
\item For all $s\in \R H^0(\Sigma;\mathcal{L}^d)$, consider the empirical measure $$\nu_s=\sum_{x\in Z_s}\delta_x,$$ where $\delta_x$ is the Dirac measure at $x$. It induces a measure $\nu_s^k$ on $\R\Sigma^k$ defined, for every continuous function $f\in C^0(\R \Sigma^k)$, by $$\nu_s^k(f)=\displaystyle\sum_{(x_1,\dots,x_k)\in Z_s^k}f(x_1,\dots,x_k).$$
\item  The curvature form $\omega$ of $(\mathcal{L},h)$ gives rise to the K\"ahler metric on $\Sigma$ which induces  a Riemannian metric on the real locus $\R\Sigma$. 
We denote by $|\dV_h|^k$ the induced Riemannian volume form on $\R\Sigma^k$. 
\item For any $1\leq a<b\leq k$, we denote by $j_{ab}:\R\Sigma^{k-1}\hookrightarrow \R \Sigma^k$  the inclusion $$(x_1,\dots,x_a,\dots,\hat{x}_b,\dots,x_k)\mapsto (x_1,\dots,x_a,\dots,x_a,\dots,x_k),$$ so that the image $j_{ab}(\R\Sigma^{k-1})$ is equal to $\{(x_1,\dots.,x_k)\in \R \Sigma^k\mid x_a=x_b\}.$
\item Finally, for any $f\in C^0(\R\Sigma^k)$, we denote by $\omega_f$ its modulus of continuity, that is $\omega_f:\epsilon\in\R_+^*\mapsto\sup\{|f(x)-f(y)|, \mathbf{d}_h(x,y)\leq\epsilon \}\in\R_+$, where $\mathbf{d}_h$ is the Riemannian distance on $\R\Sigma$ induced by the real Hermitian metric $h$.
\end{itemize}
\begin{thm}\label{equimom} 
There exists a universal positive constant $M$ such that, for any  positive real Hermitian line bundle $(\mathcal{L},h)$  over any real Riemann surface $\Sigma$,  any $k\in\mathbb{N}$ and  any $f:\R\Sigma^k\rightarrow\R$ bounded  function, the following asymptotics hold:
$$\displaystyle\frac{1}{\sqrt{d}^k}\mathbb{E}[\nu_s^k](f)=\frac{1}{\sqrt{\pi}^k}\int_{\R \Sigma^k}f|\dV_h|^k+O(\frac{\log d}{\sqrt{d}}).$$
Moreover, if $f$ is continuous, we have:
$$\displaystyle\frac{1}{\sqrt{d}^k}\mathbb{E}[\nu_s^k](f)=\frac{1}{\sqrt{\pi}^k}\int_{\R \Sigma^k}f|\dV_h|^k+\frac{M}{\sqrt{\pi}^{k-2}\sqrt{d}}\sum_{1\leq a< b\leq k}\int_{\R\Sigma^{k-1}}j_{ab}^*f|\dV_h|^{k-1}+o(\frac{1}{\sqrt{d}}).$$ 
The error term $o(\frac{1}{\sqrt{d}})$ is bounded from above by $$\norm{f}_{\infty}\big(O\big(\frac{1}{\sqrt{d}^{1+\alpha}}\big)+\omega_f\big(\frac{1}{\sqrt{d}^{\alpha}}\big)O(\frac{1}{\sqrt{d}})\big)$$
for any $\alpha\in(0,1)$, where $\omega_f(\cdot)$ is the modulus of continuity of $f$. 
Moreover, the errors $O\big(\frac{1}{\sqrt{d}^{1+\alpha}}\big)$, $O(\frac{1}{\sqrt{d}})$ do not depend on $f$. 
\end{thm} 
Choosing $f=1$ in Theorem \ref{equimom}, we obtain Theorem \ref{asymom}.\\

We also investigate  the asymptotic behaviour of the central moments.
 We  define the $k$-th central moment of the random variable $\nu_s$ as $\mathbb{E}[(\nu_s-\mathbb{E}[\nu_s])^k]$, where, for all $f\in C^0(\R\Sigma)$,
$$\mathbb{E}[(\nu_s-\mathbb{E}[\nu_s])^k](f)=\int_{s\in\R H^0(\Sigma;\mathcal{L}^d)}\big(\nu_s(f)-\mathbb{E}[\nu_s](f)\big)^k d\mu(s).$$

\begin{thm}\label{centermom}
Under the hypothesis of Theorem \ref{asymom}, for all $f:\R\Sigma\rightarrow\R$ bounded  function and any $k>2$, the following holds as $d\rightarrow\infty$: 
$$\frac{1}{\sqrt{d}^{k}}\mathbb{E}[(\nu_s-\mathbb{E}[\nu_s])^k](f)=O(\frac{\log d}{\sqrt{d}}).$$
Moreover, if $f$ is continuous, we have 
$$\frac{1}{\sqrt{d}^{k-1}}\mathbb{E}[(\nu_s-\mathbb{E}[\nu_s])^k](f)=o(1).$$
The term $o(1)$ is bounded from above by $$\norm{f}_{\infty}\big(O\big(\frac{1}{\sqrt{d}^{\alpha}}\big)+\omega_f\big(\frac{1}{\sqrt{d}^{\alpha}}\big)O(1)\big)$$
for any $\alpha\in(0,1)$, where $\omega_f(\cdot)$ is the modulus of continuity of $f$. 
Moreover, the errors $O\big(\frac{1}{\sqrt{d}^{\alpha}}\big)$ and $O(1)$ do not depend on $f$.   
\end{thm}
In particular, for $f=1$, we get $\mathbb{E}[(\#Z_s-\mathbb{E}[\# Z_s])^k]=O(\sqrt{d}^{k-1-\alpha})$.\\
  In Theorem \ref{centermom} we put the hypothesis $k>2$ in order to have   a more concise formula. In fact,  the moment $ \mathbb{E}[(\#Z_s-\mathbb{E}[\# Z_s])^k]$ has a different behavior in the case $k=1$ or $2$. For $k=1$ it is trivially $0$. 
For $k=2$,  which corresponds to the variance $\textrm{Var}(\#Z_s)$, it has been proved  that it grows as $\sqrt{d}$,  see \cite{Puchol}.

\paragraph*{Case of random polynomials}
Let $p\in\R_d[X]$ be a degree $d$ real polynomial  and $Z_p=\{x\in\R\mid p(x)=0\}$ be the real zeros of $p$. We equip $\R_d[X]$ with the Gaussian measure $\mu$ associated with the scalar product that makes $\{\sqrt{\binom{d}{k}}X^k\}_{0\leqslant k\leqslant d}$ an orthonormal basis. The probability space $(\R_d[X],\mu)$ is  called the Kostlan-Shub-Smale model.  

\begin{cor} Let $c(d)\in \R^*_+$ be any sequence of real numbers, then for any  $k>2$ and any $\alpha\in(0,1)$,  the following holds as $d\rightarrow\infty$:
\begin{itemize}
\item$\mu\big\{\big|\# Z_p-\sqrt{d}\big|>c(d)\sqrt{d}\big\}= O\big(\frac{1}{c(d)^k\sqrt{d}^{1+\alpha}}\big);$

\item for any measurable subset $A\subset \R P^1$ of positive volume: $$\mu\big\{\big|\# (Z_p\cap A)-\mathbb{E}[\#(Z_p\cap A)]\big|>c(d)\sqrt{d}\big\}= O\big(\frac{\log d}{c(d)^k\sqrt{d}}\big).$$
\end{itemize}
\end{cor}
 \begin{cor}\label{asypol} Let $p\in(\R_d[X],\mu)$ be a random Kostlan-Shub-Smale polynomial. There exists a positive constant $C$ such that  for any $k\in\mathbb{N}$  and any $\alpha\in(0,1)$ the following asymptotics hold 
 $$\mathbb{E}[\#Z_p^k]=\sqrt{d}^k+Ck(k-1)\sqrt{d}^{k-1}+O\big(\sqrt{d}^{k-1-\alpha}\big)$$
$$\mathbb{E}[|\#Z_p-\sqrt{d}|^k]=O\big(\sqrt{d}^{k-1-\alpha}\big).$$
\end{cor}
A Kostlan-Shub-Smale polynomial $p\in(\R_d[X],\mu)$  is in fact a real random section of  $\mathcal{O}_{\C P^1}(d)\rightarrow \C P^1$, where we equip $\mathcal{O}_{\C P^1}(1)$ and hence $\mathcal{O}_{\C P^1}(d)$ with the Fubini Study metric. The corollaries follow from Theorems \ref{expon}, \ref{asymom} and \ref{centermom}  and from the fact that the Fubini-Study volume of $\R P^1$ is equal to $\sqrt{\pi}$. 
The constant $C$ appearing  in  Corollary \ref{asypol} is then equal to $\frac{M\sqrt{\pi}}{2}$, where $M$ is the universal constant of Theorem \ref{equimom}. 
Corollary \ref{asypol} for $k=2$ has already been proved in \cite{dal}, in which a Central Limit Theorem for Kostlan-Shub-Smale polynomial is also shown (see also \cite{azais} for the variance of  the number of solutions of a system of random polynomials). However, also for the case of random polynomials, our proof differs from the one of \cite{dal}. 
\paragraph*{Complex case}
These techniques can be applied also in the complex case.  For all $s\in H^0(\Sigma;\mathcal{L}^d)$, let $C_s$ be the current of integration over $\{s=0\}$, that is the empirical measure $\displaystyle C_s(f)=\sum_{x\in \{s=0\}} f(x)$. 
$$\emph{How do the zeros of a random section  distribute over $\Sigma?$}$$
It is known that $\frac{1}{d}\mathbb{E}[C_s]=\omega+O(\frac{1}{d})$, where $\omega$ is the curvature form of $h$, see  \cite[Proposition 3.2]{sz}. In \cite[Theorem 3.4]{bsz}, the correlations between simultaneous zeros of a random section is proved to be indipendent of the line bundle $\mathcal{L}^d$.
The main result in the complex setting is the following:
\begin{thm}\label{complex} Let $(\mathcal{L},h)$ be a positive Hermitian line bundle over a Riemann surface $\Sigma$ and let $\omega$ be the K\"ahler curvature form of $h$. Then, for all $k\in \mathbb{N}$ and for all $f\in C^0(\Sigma^k)$, the following asymptotics hold:
$$\frac{1}{d^k}\mathbb{E}[C_s^k](f)=\int_{\Sigma^k}f\dV^k_{\omega}+\norm{f}_{\infty}O(\frac{1}{d})$$  where $\dV_{\omega}^k$ is the volume form on $\Sigma^k$ induced by $\omega$. 
\end{thm}
For $k=2$,  a more precise result has been obtained by Shiffman and Zelditch, see  \cite[Theorem 1.1]{sz2}.
\paragraph*{Idea of the proof}
The pattern of the proof is the following
$$\textrm{Theorem \ref{equimom}}\Rightarrow\textrm{Theorem \ref{centermom}}\Rightarrow\textrm{Theorem \ref{expon}}.$$
The first implication is proved by simple algebraic operations, the second one using the classical Markov inequality. Theorem \ref{equimom} is then the core of the paper.
The proof of Theorem \ref{equimom}  is of geometric nature. Using Kac-Rice formula, we write the moment $\mathbb{E}[\nu_s^k](f)$ as an integral over $\R \Sigma^k$ of the form $$\int_{\underline{x}\in\R \Sigma^k}f(\underline{x})\mathcal{R}^k_d(\underline{x}) |\dV_h|^k.$$
Here, $\mathcal{R}^k_d$ is a smooth function  defined on $\R\Sigma^k\setminus\Delta$, where $\Delta$ is the  diagonal of $\R\Sigma^k$. 
By standard techniques, we can write, for any $\underline{x}=(x_1,\dots,x_k)\in \R\Sigma^k\setminus\Delta $, 
$$\mathcal{R}_d(\underline{x})=\frac{\mathcal{N}^k_d(\underline{x})}{\mathcal{D}^k_d(\underline{x})}.$$ The functions $\mathcal{N}^k_d$ and $\mathcal{D}^k_d$ are explicit and given in Proposition \ref{df}.  In particular, using Bergman kernel estimates,  we show that  $\mathcal{R}_d^k$ locally grows  as $\sqrt{d}^k$. The denominator
$\mathcal{D}^k_d(\underline{x})$ is the normal Jacobian of the evaluation map $$ev_{\underline{x}}:s\in \R H^0(\Sigma;\mathcal{L}^d)\mapsto \big(s(x_1),\dots,s(x_k)\big) \in\R\mathcal{L}^d_{x_1}\times\cdots \R\mathcal{L}^d_{x_k}.$$
It equals the square root of the determinant of a symmetric $n\times n$ matrix whose $(i,j)$-entry is $\mathcal{K}_d(x_i,x_j)$, where $\mathcal{K}_d$ is the Bergman kernel of the line bundle $\mathcal{L}^d$. This normal Jacobian vanishes on the diagonal $\Delta$ of $\R\Sigma^k$ and this is the reason why $\mathcal{R}_d^k$ is (a priori) defined only over $\R\Sigma^k\setminus\Delta$. The main steps for studying the function $\mathcal{R}_d^k$ are the following:
\begin{enumerate}
\item Outside a neighborhood $\Delta_d$ of the diagonal $\Delta$ of size around $\frac{\log d}{\sqrt{d}}$ we have
$$\frac{1}{\sqrt{d}^k}\mathcal{R}^k_d(x_1,\dots,x_k)=\prod_{i=1}^k\frac{1}{\sqrt{d}}\mathcal{R}_d^k(x_i)+O(\frac{1}{d})=\frac{1}{\sqrt{\pi}^k}+O(\frac{1}{d})$$
where the error term is uniform in $(x_1,\dots,x_k)\in\R\Sigma^k\setminus \Delta_d$. 
This is essentially due to the fact that we can express $\mathcal{R}^k_d(x_1,\dots,x_k)$ in terms of the Bergman kernel $\mathcal{K}_d$  at points $(x_i,x_j)$, $i,j\in\{1,\dots,k\}$, and that $\frac{1}{d}\mathcal{K}_d(x_i,x_j)$ uniformly goes to zero as $\mathbf{d}_h(x_i,x_j)>\frac{\log d}{\sqrt{d}}$, see Proposition \ref{sim2}. 
\item  Using Olver multispaces we are able  to extend the  function $\mathcal{R}^k_d$ over all the \emph{compact} manifold $\R\Sigma^k$. 
Then, a  careful analysis of $\mathcal{R}^k_d$ in a neighborhood of the diagonal and the compactness of $\R\Sigma^k$ will give us an  uniform boundedness of $\frac{1}{\sqrt{d}^k}\mathcal{R}^k_d$, that is there exists a constant $C$ such that $\frac{1}{\sqrt{d}^k}\mathcal{R}^k_d(\underline{x})<C$ for every $\underline{x}\in\R\Sigma^k$ and every $d$. This is the content of   Theorem \ref{boundabove}.
\item At this level we are able to prove that  $$\frac{1}{\sqrt{d}^k}\int_{\R \Sigma^k}f\mathcal{R}^k_d |\dV_h|^k=\frac{1}{\sqrt{\pi}^k}\int_{\R\Sigma^k}f|\dV_h|^k+\norm{f}_{\infty}O(\frac{\log d}{\sqrt{d}}).$$
The first term of the right hand side is given by integrating over the complement of the neighborhood $\Delta_d$ of the diagonal  and by using point 1. The error term is given by the integral over $\Delta_d$. This integral  is bounded from above by the product of the volume of $\Delta_d$, that is a $O(\frac{\log d}{\sqrt{d}})$, times the infinity norms of  $f$ and of $\frac{1}{\sqrt{d}^k}\mathcal{R}^k_d(\underline{x})$. The latter  is finite thanks to point 2 above.
\item Finally, a more careful analysis of $\mathcal{R}^k_d$ in a neighborhood of the diagonal  will give us also the second term in the asymptotic expansion of Theorem \ref{equimom}. For  this, we introduce some subsets $U_d^{a,b}$ of $\R \Sigma^k$,  $1\leq a<b\leq k$, which, roughly speaking, are the set of $(x_1,\dots,x_k)$ such that the distance between every pair  of points $(x_i,x_j)$, $i\neq j\in\{1,\dots,k\}$, is bigger than $\frac{\log d}{\sqrt{d}}$, except at most the pair $(x_a,x_b)$, see Definition  \ref{open1}. With similar techniques as in the point 1 above, we are able to estimate $\mathcal{R}_d^k$ is these subsets, see Propositions \ref{sim2} and \ref{simple}.
\end{enumerate}
The proof of the complex case, namely Theorem \ref{complex}, follows the same lines.
\paragraph*{Organization of the paper}
 In Section \ref{generality} we introduce our setting and the main  tools such as Bergman kernel estimates in normal coordinates. 
 In Section \ref{ransec} we  write the moment $\mathbb{E}[\nu_s^k](f)$ as an integral $$\int_{\underline{x}\in\R \Sigma^k}f(\underline{x})\mathcal{R}^k_d(\underline{x}) |\dV_h|^k. $$  This is done by introducing an incidence manifold and using the coarea formula.
Far from the diagonal Bergman kernel estimates gives us the asymptotic behaviour of $\mathcal{R}^k_d$, see Proposition \ref{sim2}.
 The goal of  Section \ref{pfoo} is to  prove Theorem \ref{expon}. It is a direct consequence
of the computation of central moments, Theorem \ref{centermom}, which is implied
by Theorem \ref{equimom}.  Theorem \ref{equimom} is proved in this section, admitting   a boundedness result, Theorem \ref{boundabove}, and an asymptotic expansion result, Proposition \ref{simple}. \\
 In Section \ref{olvtec} we prove Theorem \ref{boundabove} and Proposition \ref{simple}. This is the core of the paper. Olver multispaces and divided differences coordinates will  play a crucial role in the proof of this theorem.
In Section \ref{ccase} we discuss the complex case, giving  Theorem \ref{complex}. The proof in the complex case follows the lines of  the real one.

\section{Definitions and main tools}\label{generality}
\subsection{Framework}\label{generality1}
In this section we introduce our  setting, which
  is the same as in \cite{anc,gwexp,gw2}, but  we  restrict ourself to the one dimensional case.\\
\begin{itemize}
 \item  Let $\Sigma$ be a smooth real compact Riemann surface, that is a smooth Riemann surface equipped with an anti-holomorphic involution $c_{\Sigma}$. We denote by
$\mathbb{R}\Sigma=\Fix(c_{\Sigma})$ its real locus. 
 \item Let $\mathcal{L}$  be a real ample line bundle over $\Sigma$, that is an ample holomorphic line bundle $p:\mathcal{L}\rightarrow \Sigma$  equipped with an anti-holomorphic involution $c_L$ such that $p\circ c_{\mathcal{L}}=c_{\Sigma}\circ p$ and $c_L$ is complex-antilinear in the fibers. 
\item A real Hermitian metric on $\mathcal{L}$ is a Hermitian metric $h$ such that $c_{\mathcal{L}}^*h=\bar{h}_L$.
We equip $\mathcal{L}$ with a real Hermitian metric $h$ of positive curvature $\frac{i}{2\pi}\partial\bar{\partial}\phi=\omega\in\Omega^{(1,1)}(\Sigma,\R)$. The \emph{local potential} $\phi$ equals $-\log h(e_{\mathcal{L}},e_{\mathcal{L}})$ where $e_{\mathcal{L}}$ is any local holomorphic trivialization of $\mathcal{L}$.  
  \item The curvature form induces a K\"ahler metric $\omega(\cdot,i\cdot)$ on $\Sigma$ which restricts to a Riemannian metric over $\R \Sigma$. We denote  the Riemannian length form by $|\dV_h|$. We will denote also the Riemannian volume form on $\R\Sigma^k$ by $|\dV_h|^k$.
  \item Let $dx=\frac{\omega}{\int_{\Sigma}\omega}$ be the normalized volume form on $\Sigma$.
  \end{itemize}
\subsubsection{The Gaussian measure on $\R H^0(\Sigma;\mathcal{L}^d)$}  We denote by $\R H^0(\Sigma;\mathcal{L})$ the real vector space of real global sections of $\mathcal{L}$, i.e. sections $s\in H^0(\Sigma;\mathcal{L})$ such that $s \circ c_{\Sigma}=c_{\mathcal{L}}\circ s$. The Hermitian metric $h$ induces a Hermitian metric $h^d$ on $\mathcal{L}^d$ for every integer $d>0$ and also a $L^2$-Hermitian product on the space $H^0(\Sigma;\mathcal{L}^d)$ of global holomorphic sections of $\mathcal{L}^d$ denoted by $\langle\cdot,\cdot \rangle$ and  defined by 
$$\langle\alpha,\beta\rangle=\int_{\Sigma}h^d(\alpha,\beta)dx$$
for any $\alpha,\beta$ in $H^0(\Sigma;\mathcal{L}^d)$.\\The $L^2$-Hermitian product $\langle\cdot,\cdot\rangle$ on $H^0(\Sigma;\mathcal{L}^d)$ restricts to a $L^2$-scalar product on $\R H^0(\Sigma;\mathcal{L}^d)$, also denoted by $\langle\cdot,\cdot\rangle$. Then  we have a natural Gaussian measure on $\R H^0(\Sigma;\mathcal{L}^d)$ defined by

$$\mu(A)=\frac{1}{\sqrt{\pi}^{N_d}}\int_Ae^{-\parallel s\parallel^2}ds$$
for any open subset $A\subset \R H^0(\Sigma;\mathcal{L}^d)$ where $ds$ is the Lebesgue measure associated with $\langle\cdot,\cdot\rangle$ and $N_d=\dim_{\C}H^0(\Sigma;\mathcal{L}^d)=\dim_{\R}\R H^0(\Sigma;\mathcal{L}^d)$.\\

\subsection{Bergman kernel}\label{bochnerco}
In this section we recall some asymptotic estimates of the Bergman kernel related to Hermitian line bundles, see \cite{ber1,ber2,ma1,ma2}.  
Let $(\mathcal{L},h)$ be a real Hermitian line bundle of positive curvature $\omega$ over a real projective manifold $X$.
We denote by $L^2(X;\mathcal{L}^d)$ the space of square integrable global sections of $\mathcal{L}^d$ and by $\R L^2(X;\mathcal{L}^d)$ the space of real square integrable global sections. The orthogonal projection from $\R L^2(X;\mathcal{L}^d)$ onto $\R H^0(X;\mathcal{L}^d)$ admits a Schwartz kernel $\mathcal{K}_d$. It means that there exists a unique section $\mathcal{K}_d$ of the bundle $\mathcal{L}^d\boxtimes (\mathcal{L}^d)^*$ over $X\times X$ such that, for any $s\in \R L^2(X;\mathcal{L}^d)$ the projection of $s$ onto $\R H^0(X;\mathcal{L}^d)$ is given by $$x\mapsto \int_{y\in X}\mathcal{K}_d(x,y)(s(y))dx.$$
This Schwartz kernel $\mathcal{K}_d(x,y)$ is called the \emph{Bergman kernel} of $\mathcal{L}^d$.  
\begin{oss} Let $N_d$ be the dimension of $\R H^0(X;\mathcal{L}^d)$ and $(s_{1},\dots,s_{N_d})$ be any orthonormal basis of $\R H^0(X;\mathcal{L}^d)$. Then
$\mathcal{K}_d(x,y)$ is equal to $$\sum_{i=1}^{N_d}s_{i}(x)\otimes s_{i}(y)^*.$$
\end{oss}
\subsubsection{Exponential decay} We recall the following  theorem of \cite{ma1}. 
\begin{thm}\label{fardiag}
Let $(\mathcal{L},h)\rightarrow X$ be a Hermitian positive line bundle over a complex manifold $X$ of dimension $n$.
There exist $C'>0$ and $d_0\in \mathbb{N}$ such that, for any $m\in \mathbb{N}$, there exists $C_m>0$ such that $\forall d\geqslant d_0$, 
$\forall x,y\in X$

$$\norm{\mathcal{K}_d(x,y)}_{C^m}\leqslant C_md^{n+\frac{m}{2}}\exp(-C'\mathbf{d}_h(x,y)\sqrt{d})$$
where $\mathbf{d}_h$ is the geodesic distance in $X$ induced by $h$ and the norms of the derivates of $K_d$ are also induced by $h$.
\end{thm}

\subsubsection{Normal coordinates}\label{normal}   In this section we define our preferred coordinates and trivializations to which we refer to as normal coordinates. For a more detailed introduction of these coordinates, see \cite[Section 3.1]{Puchol}.\\
Let $(\mathcal{L},h)$ be a real Hermitian line bundle of positive curvature $\omega$ over a  projective manifold $X$. Let $x$ be a point in $X$ and $U$ be a small  neighborhood of $x$. The \emph{normal coordinates around $x$} is the exponential coordinates  $\exp_x:z\in V\subset T_x X\mapsto U\subset X$  (which, in general, is not holomorphic) \emph{together with} the local trivialization $\mathcal{L}^d_x\times U\simeq \mathcal{L}^d\mid_{U}$ of $\mathcal{L}^d$ given, at any point  $\exp_x(z)\in U$, by the parallel transport (induced by the Chern connection of $\mathcal{L}^d$) of the fiber $\mathcal{L}^d_x$ along the geodesic $t\mapsto \exp_x(tz)$. For any $d>0$, the \emph{scaled normal coordinates} around $x$ is the composition of the map $T\mapsto z=\frac{T}{\sqrt{d}}$ with the normal coordinate around $x$.\\
In the presence of a real structure $c_X$,  for any $x\in\R X$, the restriction of the exponential map $\exp_x\mid_{T_x\R\Sigma}$ coincides with the exponential map associated with the Riemannian metric  $\R X$ induced by $\omega$. The normal coordinates around a real point $x\in\R X$ is the (real) exponential coordinates $\exp_x:z\in V\subset T_x \R X\mapsto U\subset \R X$  \emph{together with} the local trivialization $\R \mathcal{L}^d_x\times U\simeq \R\mathcal{L}^d\mid_{U}$ of $\R \mathcal{L}^d$  around $x$ given, at any point  $\exp_x(z)\in U$, by the parallel transport  of the fiber $\R \mathcal{L}^d_x$ along the geodesic $t\mapsto \exp_x(tz)$.

\subsubsection{Near diagonal estimate} The following theorem says that in the (scaled) normal coordinates $U_x$ around a point $x\in X$ the geometry of $\mathcal{L}^d_{\mid U_x}\rightarrow U_x$ looks like the geometry of the Bargmann-Fock space (see Section \ref{local2}), at least in a ball of size $B(x;\frac{R \log d}{\sqrt{d}})$ for large $d$ and  any fixed $R>0$. The following theorem  is the main theorem  of \cite{daima} (see also \cite[Theorem 4.18]{ma2},  \cite{ber2,ber1}).

\begin{thm}\label{neardiag} Let $(\mathcal{L},h)\rightarrow X$ be a Hermitian positive line bundle over a complex manifold $X$ of dimension $n$. Fix $m\in \mathbb{N}$ and $R>0$. Then, for any $\alpha\in (0,1)$  any $x\in X$ and  any $z,w\in B(x,R\frac{\log d}{\sqrt{d}})$, one have

$$\norm{ \mathcal{K}_d(z,w)- (\frac{d}{\pi})^ne^{ \frac{-d\norm{z-w}^2}{2}}}_{C^m}=O(d^{n-\alpha})$$
in the  normal coordinates around $x$.
Here, $\parallel\cdot\parallel$ is the norm on $T_xX$ induced by $h$ and $\parallel\cdot\parallel_{C^m}$ is the $C^m$-norm on $(\mathcal{L}^d)_x$ induced by $h$. The error term   does not depend on $x,z,w$ but only on $m$, $R$ and $\alpha$.
\end{thm}
If $X$ is a real algebraic variety,  we obtain: 
\begin{thm}\label{neardiagreal}
Let $(\mathcal{L},h)\rightarrow X$ be a real  Hermitian positive line bundle over a real algebraic variety $X$ of dimension $n$.
 Fix $m\in \mathbb{N}$ and $R>0$.  For any $\alpha\in (0,1)$,  any $x\in \R X$ and  any $z,w\in B(x,R\frac{\log d}{\sqrt{d}})$, one have

$$\norm{\mathcal{K}_d(z,w)-(\frac{d}{\pi})^ne^{\frac{-d\norm{z-w}^2}{2}}}_{C^m}=O(d^{n-\alpha})$$
in the  normal coordinates around $x$.
Here,  $\parallel\cdot\parallel_{C^m}$ is the $C^m$-norm on $(\mathcal{L}^d)_x$ induced by $h$. The error term   does not depend on $x,z,w$ but only on $m$, $R$, and $\alpha$.
\end{thm}
\begin{proof}
This is  Theorem \ref{neardiag}, restricting everything to the real locus $\R X$ of $X$. 
\end{proof}
\subsubsection{Scaled Bergman kernel} These theorems suggest us the following
\begin{defn}\label{scalbegmkern}  Fix a real point $x\in\R X$.
We define the \emph{scaled Bergman kernel} by $K_d(Z,W)=\frac{1}{d^n}\mathcal{K}_d(\frac{Z}{\sqrt{d}},\frac{W}{\sqrt{d}})$, where $Z,W$ are the scaled (real) normal coordinates around  $x$  (that are $Z=\sqrt{d}z$ and $W=\sqrt{d}w$) and the \emph{local Bergman kernel} by $K_{\C^n}(Z,W)=\frac{1}{\pi^n}e^{\frac{-\norm{Z-W}^2}{2}}$.
\end{defn}
In particular,  Theorem \ref{neardiagreal} can be written as: 

\begin{thm}\label{berglocal}
 Fix $k\in \mathbb{N}$ and $R>0$. For any $\alpha\in (0,1)$,  for any $x\in \R X$ and  any $Z,W\in B(x,R \log d)$ we have
 $$\norm{K_d(Z,W)-K_{\C^n}(Z,W)}_{C^m}=O(\frac{1}{d^{\alpha}})$$
 in the scaled normal coordinates around $x$. Moreover the error term only depends on $m$, $R$ and $\alpha$.
 \end{thm}

\section{Incidence manifold and density function}\label{ransec}
Throughout  this section, we will denote by $(\mathcal{L},h)$  a real ample Hermitian  line bundle over a real Riemann surface $\Sigma$, see Section \ref{generality1}.\\
In the first susbection  we define some modified measures $\tilde{\nu}_s^k$.
 Then, we introduce an incidence manifold and  use the coarea formula  to write the \emph{modified moments} $\mathbb{E}[\tilde{\nu}_s^k]$ as an integral of a density function $\mathcal{R}_d^k$ over $\R\Sigma^k$. This is done in the second subsection. The estimate of the density function $\mathcal{R}_d^k$ is the most  important step in  Theorem \ref{equimom}.
We  start the study of $\mathcal{R}_d^k$ by writing it  as a fraction (Proposition \ref{df}) and by giving an off-diagonal estimate (Proposition \ref{sim2}).

 \subsection{Zeros of random real sections}\label{modmoment}
Let $s\in\R H^0(\Sigma;\mathcal{L}^d)$ be any real section of $\mathcal{L}^d$, that is any global holomorphic section of $\mathcal{L}^d$ such that $s\circ c_{\Sigma}=c_{\mathcal{L}}\circ s$, where $c_{\mathcal{L}}$ and $c_{\Sigma}$ are the real structures of $\mathcal{L}$ and $\Sigma$, see Section \ref{generality1}.
\begin{defn}\label{modmom} 
\begin{itemize}
\item To any non-zero section $s\in\R H^0(\Sigma;\mathcal{L}^d)$ we  associate the following empirical measure $\nu_s=\sum_{x\in Z_s}\delta_x$
where $Z_s=\R\Sigma\cap \{s=0\}$ is the real vanishing locus of $s$ and $\delta_x$ is the Dirac measure at $x\in Z_s$. It induces an empirical measure $\nu_s^k$ on $\R\Sigma^k$ 
for any $k\in\mathbb{N}^*$,
defined, for any $f\in C^0(\R\Sigma^k)$, by 
$$\nu_s^k(f)=\sum_{(x_1,\dots,x_k)\in Z^k_s}f(x_1,\dots,x_k).$$
\item For any $s\in\R H^0(\Sigma;\mathcal{L}^d)$ and any $k\in\mathbb{N}$ we define $\tilde{\nu}^k_s$ to be the following modified empirical measure: $$\tilde{\nu}^k_s(f)=\sum_{\substack{(x_1,\dots,x_k)\in Z^k_s \\ x_i\neq x_j}}f(x_1,\dots,x_k).$$
\end{itemize}
\end{defn}
 
\subsubsection{Partitions of $k$ elements} Let $\mathcal{P}_k$ be the set of all the partitions of $\{1,\dots,k\}$, that means the set of all $$I=\big\{\{1_1,\dots,1_{k_1}\},\dots,\{m_1,\dots,m_{k_m}\}\big\}$$ such that $\{1_1,\dots,1_{k_1}\}\sqcup\dots\sqcup\{m_1,\dots,m_{k_m}\}=\{1,\dots,k\}$. 
For any $I\in\mathcal{P}_k$, we write $(x_1,\dots,x_k)\in Z^k_{s,I}$ if and only if $(x_1,\dots,x_k)\in Z^k_s$ and $x_{i_l}=x_{i_h}$ for any $l,h\in\{1,\dots,k_i\}$ and $x_{i_l}\neq x_{j_h}$ if $i\neq j\in\{1,...,m\}$.
 We then write 
 $$\nu_s^k(f)=\sum_{I\in\mathcal{P}_k}\sum_{(x_1,\dots,x_k)\in Z^k_{s,I}}f(x_1,\dots,x_k).$$
Now, we denote by $\widetilde{\mathcal{P}}_k=\mathcal{P}_k\setminus \big\{\{1\},\{2\},\dots,\{k\}\big\}$ so that 
 $$\nu_s^k(f)=\sum_{I\in\tilde{\mathcal{P}}_k}\sum_{(x_1,\dots,x_k)\in Z^k_{s,I}}f(x_1,\dots,x_k)+\tilde{\nu}_s^k(f).$$
 
For any partition $I=\big\{\{1_1,\dots,1_{k_1}\},\dots,\{m_1,\dots,m_{k_m}\}\big\}
\in\widetilde{\mathcal{P}}_k$, we denote by $j_I:\R\Sigma^m\rightarrow\R\Sigma^k$
  the inclusion defined by $(x_1,\dots,x_m)\mapsto (\tilde{x}_1,\dots,\tilde{x}_k)$ where, for any $i$, $\tilde{x}_{i_j}=x_i$ for any 
 $j=1,\dots,k_j$.
 The following proposition is a direct consequence of the definitions just given.
 \begin{prop}\label{difference} For every real section $s\in\R H^0(\Sigma;\mathcal{L}^d)$, every integer $k\in\mathbb{N}$ and every  partition $I=\big\{\{1_1,\dots,1_{k_1}\},\dots,\{m_1,\dots,m_{k_m}\}\big\}\in \mathcal{P}_k$, we have
 $$\sum_{(x_1,\dots,x_k)\in Z^k_{s,I}}f(x_1,\dots,x_k)=\tilde{\nu}^m_s(j^*_If)$$so that 
 $$\nu_s^k(f)=\tilde{\nu}^k_s(f)+
 \sum_{I\in\widetilde{\mathcal{P}}_k}\tilde{\nu}^{m_I}_s(j^*_If)$$
 
\end{prop}
\subsubsection{Modified moments} To understand $\nu_s^k(f)$, it is enough, in fact equivalent, to  understand  $\tilde{\nu}_s^k(f)$. 
We will study the expectation of this new random variable, namely
$$\mathbb{E}[\tilde{\nu}^k_s](f)=\int_{s\in\R H^0(\Sigma;\mathcal{L}^d)}\sum_{\substack{(x_1,\dots,x_k)\in Z^k_s \\ x_i\neq x_j}}f(x_1,\dots,x_k)d\mu(s)$$
We call $\mathbb{E}[\tilde{\nu}^k_s]$ \emph{the modified $k$-moment} of $\nu_s$.

\subsection{Incidence manifold and density function $\mathcal{R}_d^k$}\label{incsect}
We introduce an incidence manifold $\mathcal{I}$ that comes equipped with two projections respectively to $\R H^0(\Sigma;\mathcal{L}^d)$ and to $\R \Sigma^k$. Following \cite{ss}  (see also \cite{gw2}, \cite{anc}) we will apply the coarea formula (see \cite[Lemma 3.2.3]{fed} or \cite[Theorem 1]{ss}) to these two projections to write the modified moment $\mathbb{E}[\tilde{\nu}_s^k]$ as an integral over $\R \Sigma^k$. The resulting formula is often known as  Kac-Rice formula, see for example \cite[Theorem 3.2]{azais2}. 

\subsubsection{Incidence manifold} Let $\Delta$ be the diagonal of $\R\Sigma^k$ defined by 
$$\Delta=\{(x_1,\dots,x_k)\in\R\Sigma^k | \exists \hspace{0.5mm} i\neq j, \hspace{1mm} x_i=x_j\}$$
and set
$$\mathcal{I}=\{(s,x_1,\dots,x_k)\in\R H^0(\Sigma;\mathcal{L}^d)\times(\R\Sigma^k\setminus\Delta)\mid s(x_i)=0 \hspace{2mm}  i=1,\dots,k\}.$$
We denote by $\pi_{\Sigma}$ (resp. $\pi_H$) the  projection $\mathcal{I}\rightarrow \R\Sigma^k$
(resp. $\mathcal{I}\rightarrow \R H^0(\Sigma;\mathcal{L}^d)$).
\begin{prop}\label{smoothman} Let $\mathcal{L}$ be a positive real line bundle over  a real Riemann surface $\Sigma$. Then, for large $d\in\mathbb{N}$, the set $\mathcal{I}$ is a smooth manifold. We call $\mathcal{I}$ the \emph{incidence manifold}.
\end{prop}
\begin{proof} Consider the map $\R H^0(\Sigma;\mathcal{L}^d)\times(\R\Sigma^k\setminus\Delta)\rightarrow \R \mathcal{L}^d\times\dots\times \R \mathcal{L}^d$ defined by $(s,x_1,\dots,x_k)\mapsto(s(x_1),\dots,s(x_k))$.  We have to prove that $0$ is a regular value. 
The derivative of this map is $$(\dot{s},\dot{x}_1,\dots,\dot{x}_k)\mapsto (\dot{s}(x_1)+\nabla_{\dot{x}_1}s(x_1),\dots,\dot{s}(x_k)+\nabla_{\dot{x}_k}s(x_k)).$$ 
By the positivity  of $\mathcal{L}$ and by Riemann-Roch theorem, there exists $d_0$ such that for any $d\geqslant d_0$ and any $(x_1,\dots,x_k)\in\R\Sigma^k\setminus\Delta$ we can find sections $\dot{s}_1,\dots,\dot{s}_k$ such that $\dot{s}_i(x_i)\neq 0$ and $\dot{s}_i(x_j)=0$ for $i\neq j$. This implies that $0$ is a regular value. 
\end{proof}
\subsubsection{The density function $\mathcal{R}^k_d$} 
\begin{defn}
The \emph{normal jacobian} $\Jac_Nu$ of a submersion $u:M\rightarrow N$ between Riemannian manifolds is the determinant of the differential of the map restricted to the orthogonal of its kernel. Equivalently, if $du_p$ is the differential of $u$ at $p$, then the normal jacobian is equal to $\sqrt{\det(du_pdu_p^*)}$,
where $du_p^*$ is the adjoint of $du_p$ with respect to the scalar products on $T_pM$ and $T_{u(p)}N$.
\end{defn}
Recall that we have defined the modified moment  to be $\mathbb{E}[\tilde{\nu}^k_s]$, where $\tilde{\nu}^k_s$ was defined in Definition \ref{modmom}.
\begin{prop}\label{denf} Let $(\mathcal{L},h)$ be a real Hermitian line bundle over a real Riemann surface $\Sigma$. Then, for large $d\in\mathbb{N}$, we have $$\mathbb{E}[\tilde{\nu}^k_s](f)=\displaystyle\int_{\underline{x}\in\R\Sigma^k\setminus\Delta}f(\underline{x}) \mathcal{R}^k_d(\underline{x})|\dV_h|^k$$ where $\tilde{\nu}^k_s$ was defined in Definition \ref{modmom} and
$$\mathcal{R}^k_d(\underline{x})=\int_{\pi_{\Sigma}^{-1}(\underline{x})}\frac{1}{\mid Jac_N(\pi_{\Sigma})\mid}d\mu_{\mid\pi_{\Sigma}^{-1}(\underline{x})}.$$
\end{prop}
\begin{proof}
We consider $$\mathbb{E}[\tilde{\nu}^k_s](f)=\displaystyle\int_{s\in\R H^0(\Sigma;\mathcal{L}^d)}\sum_{\substack{(x_1,\dots,x_k)\in Z^k_s \\ x_i\neq x_j}}f(x_1,\dots,x_k)d\mu(s).$$
Using $\pi_H$ we pull-back the integral over $\mathcal{I}$. On $\mathcal{I}$ we put the (singular) metric $\pi_H^*\langle\cdot,\cdot\rangle$.
We have $\mathbb{E}[\tilde{\nu}^k_s](f)=\int_{\mathcal{I}}(\pi_{\Sigma}^*f)(s,x_1,\dots,x_k)(\pi_H^*d\mu)(s,x_1,\dots,x_k)$.
The coarea formula  (see \cite[Lemma 3.2.3]{fed} or \cite[Theorem 1]{ss}) applied to the submersion $\pi_{\Sigma}:\mathcal{I}\rightarrow \R\Sigma^k$ gives us the result.
\end{proof}
\begin{defn} We call  \emph{density function} the function $\mathcal{R}^k_d:\R\Sigma^k\setminus\Delta\rightarrow \R$ that appeared in Proposition \ref{denf}.
\end{defn}
\begin{prop}\label{symact}  For any permutations $\sigma\in\mathcal{S}_k$ we have $\mathcal{R}^k_d\circ\sigma=\mathcal{R}^k_d$.
\end{prop}
\begin{proof}
The symmetric group $\mathcal{S}_k$ acts by isometries on $\R \Sigma^k$ and on $\mathcal{I}$, this implies the result.
\end{proof}
\begin{defn}\label{evnaive} For $\underline{x}=(x_1,\dots,x_k)\in\R\Sigma^k\setminus\Delta$ we denote by $\R H^0_{\underline{x}}=\R H^0_{x_1,\dots,x_k}$ the subspace of real  global sections $s\in\R H^0(\Sigma;\mathcal{L}^d)$ such that $s(x_i)=0$ for any $i\in\{1,\dots,k\}$, that is the kernel of the evaluation   map $$ev_{\underline{x}}:\R H^0(\Sigma;\mathcal{L}^d)\rightarrow (\R \mathcal{L}^d)_{x_1}\times\dots\times(\R \mathcal{L}^d)_{x_k}$$
defined by $s\mapsto \big(s(x_1),\dots,s(x_k)\big).$
\end{defn}
\begin{prop}\label{df} Let $\mathcal{R}^k_d$ be the density function defined in Proposition $\ref{denf}$. Then, for any $\underline{x}=(x_1,\dots,x_k)\in\R\Sigma^k\setminus\Delta$, we have  $$\mathcal{R}_d(\underline{x})=\frac{\mathcal{N}^k_d(\underline{x})}{\mathcal{D}^k_d(\underline{x})}$$ where 
$$\mathcal{N}^k_d(\underline{x})=\int_{s\in\R H^0_{\underline{x}}}\norm{\nabla s(x_1)}\cdots\norm{\nabla s(x_k)} d\mu_{\mid \R H^0_{\underline{x}}}(s)$$
and
$$\mathcal{D}^k_d(\underline{x})=|\Jac_N(ev_{\underline{x}})|.$$
Here we have denoted by $\norm{\cdot}$ the norms induced by the Hermitian metric $h$.
\end{prop}
\begin{proof}  This is the Kac-Rice formula (see \cite[Theorem 3.2]{azais2}). Let us sketch the proof.\\
Fix $\underline{x}=(x_1,\dots,x_k)\in\R\Sigma^k\setminus\Delta$, by Proposition \ref{denf} we have
$$\mathcal{R}^k_d(\underline{x})=\int_{\pi_{\Sigma}^{-1}(\underline{x})}\frac{1}{\mid Jac_N(\pi_{\Sigma})\mid}d\mu_{\mid\pi_{\Sigma}^{-1}(\underline{x})}.$$
As $\pi_H$ is (almost everywhere) a local isometry, we can pushforward this integral onto $\R H^0(\Sigma;\mathcal{L}^d)$ and we obtain $$\mathcal{R}^k_d(\underline{x})=\int_{\R H^0_{\underline{x}}}\frac{1}{\mid (\pi_H^{-1})^*Jac_N(\pi_{\Sigma})\mid}d\mu_{\mid\R H^0_{\underline{x}}}.$$
Let $(s,x_1,\dots,x_k)$ be a point in $\pi_{\Sigma}^{-1}(x_1,\dots,x_k)\subset\mathcal{I}$ such that $\nabla s(x_i)$ is invertible for any $i\in\{1,...,k\}$ and 
let $(\dot{s},\dot{x}_1,\dots,\dot{x}_k)$ be any tangent vector of $(s,x_1,\dots,x_k)$. We have  $d\pi_{\Sigma}\cdot(\dot{s},\dot{x}_1,\dots,\dot{x}_k)=(\dot{x}_1,\dots,\dot{x}_k)$.
Remember that, by definition, the norm $\norm{(\dot{s},\dot{x}_1,\dots,\dot{x}_k)}$ on $\mathcal{I}$ equals $\norm{\dot{s}}_{L^2}$ and also that for every $i=1,\dots,k$ we have
$\dot{s}(x_i)+\nabla_{\dot{x}_i}s(x_i)=0$, so that 
$$d\pi_{\Sigma}\cdot(\dot{s},\dot{x}_1,\dots,\dot{x}_k)=(\dot{x}_1,\dots,\dot{x}_k)=(-\nabla s(x_1)^{-1}\circ \dot{s}(x_1),\dots,-\nabla s(x_k)^{-1}\circ \dot{s}(x_k)).$$
 Consider the map $$B:T_{x_1}\R\Sigma\times\dots\times T_{x_k}\R\Sigma\rightarrow (\R \mathcal{L}^d)_{x_1}\times\dots\times(\R \mathcal{L}^d)_{x_k}$$
 defined by $B(\dot{x}_1,\dots,\dot{x}_k)=\big(\nabla_{\dot{x}_1}s(x_1),\dots,\nabla_{\dot{x}_k}s(x_k)\big)$.
 We remark that $B$ is a diagonal map so that $\Jac B=\parallel\nabla s(x_1)\parallel\dots\parallel\nabla s(x_k)\parallel$.
We have $$d\pi_{\Sigma}=B^{-1}\circ ev_{x_1,\dots,x_k} \circ d\pi_H$$ and this implies $$d\pi_{\Sigma}\mid_{(d\pi_H^*\ker ev_{x_1,\dots,x_k})^{\perp}}=B^{-1}\circ ev_{x_1,\dots,x_k}\mid_{(\ker ev_{x_1,\dots,x_k})^{\perp}}\circ d\pi_H.$$
It follows that $$\Jac(d\pi_{\Sigma}\mid_{(d\pi_H^*\ker ev_{x_1,\dots,x_k})^{\perp}})=\Jac(B^{-1})\Jac(ev_{x_1,\dots,x_k}
\mid_{(\ker ev_{x_1,\dots,x_k})^{\perp}})\Jac \pi_H.$$
Now, we have that $\Jac(B^{-1})=\Jac(B)^{-1}$, $\Jac\big(ev_{x_1,\dots,x_k}\mid_{(\ker ev_{x_1,\dots,x_k})^{\perp}}\big)=\Jac_N(ev_{x_1,\dots,x_k})$ and that  the normal Jacobian of $\pi_H$ is equal to $1$ as $\pi_H$ is a local isometry.
 We conclude by integrating over $\R H^0_{\underline{x}}$ and observing that the set of $s\in\R H^0_{\underline{x}}$ such that $\nabla s(x_i)=0$ for some $i\in\{1,...,k\}$ has measure zero. 
\end{proof}
\begin{prop}\label{jet}  Let $\mathcal{L}$ be a positive real Hermitian line bundle over  a real Riemann surface $\Sigma$ and $\underline{x}=(x_1,\dots,x_k)$ be a point in $\R\Sigma^k\setminus\Delta$. Then $$\Jac_N(ev_{\underline{x}})=\sqrt{\det\big(\norm{\mathcal{K}_d(x_i,x_j)}\big)_{(i,j)}}$$ where $ev_{\underline{x}}$ is the evaluation map defined in Definition \ref{evnaive}.
\end{prop}
\begin{proof} We fix an orthonormal basis $\{s_1,\dots,s_{N_d}\}$ of $\R H^0(\Sigma;\mathcal{L}^d)$.  The Bergman kernel of the orthogonal projection onto $\R H^0(\Sigma;\mathcal{L}^d)$ is $$\mathcal{K}_d(x,y)=\sum_{i=1}^{N_d}s_i(x)\otimes s_i(y)^*.$$
We also fix a unit vector $e_i$ over each fiber $\R \mathcal{L}_{x_i}^d$, $i\in\{1,\dots,k\}$. The matrix associated with  $ev_{\underline{x}}$ with respect to these basis equals  $M=\big(h^d(s_l(x_i),e_i)\big)_{i,l}.$
Now, the matrix $MM^*$ is exactly the symmetric matrix $\big(\norm{\mathcal{K}_d(x_i,x_j)}\big)_{i,j}$. We conclude by taking the square root of the determinant of  $MM^*$, which  is exactly $\Jac_N(ev_{\underline{x}})$.
\end{proof}
\begin{prop}\label{pos}  Let $\mathcal{L}$ be a positive real Hermitian line bundle over  a real Riemann surface $\Sigma$. Then there exists an integer $d_{\mathcal{L}}\in\mathbb{N}$ such that  for any $(x_1,\dots,x_k)\in\R\Sigma^k\setminus\Delta$ and $d\geqslant d_{\mathcal{L}}$ the map
$$ev_{\underline{x}}:\R H^0(\Sigma;\mathcal{L}^d)\rightarrow \R \mathcal{L}_{x_1}^d\times\dots\times \R \mathcal{L}_{x_k}^d$$  defined by $s\mapsto (s(x_1),\dots,s(x_k))$ is surjective.
\end{prop}
\begin{proof}  Since $\mathcal{L}$  is positive, there exists $d_{\mathcal{L}}$ such that  for  $d\geq d_{\mathcal{L}}$ and  for every point $(x_1,\dots,x_k)\in\R\Sigma^k\setminus\Delta$, we have $H^1\big(\Sigma;\mathcal{L}^d(-\sum_{j=1}^k x_j)\big)=0$. From the Riemann-Roch theorem we then know that for $i\in\{1,\dots,k\}$ $$\dim H^0\big(\Sigma;\mathcal{L}^d(-\sum_{j=1}^k x_j)\big)=\dim H^0\big(\Sigma;\mathcal{L}^d(-\sum_{j\neq i}x_j)\big)-1.$$ This proves that there are  sections $s_i$ such that $s_i(x_i)\neq 0$ and $s_i(x_j)=0$ for $i\in\{1,\dots,k\}$. The surjectivity of $ev_{\underline{x}}$ follows.
\end{proof}
\begin{cor} There exists a positive integer $d_0\in\mathbb{N}$ such that for any $d\geqslant d_0$ and for any $(x_1,\dots,x_k)\in\R\Sigma^k\setminus\Delta$,
$$\mathcal{D}_d^k(x_1,\dots,x_k)>0.$$
\end{cor}
\begin{proof}
From Proposition \ref{df} we have $\mathcal{D}_d^k(\underline{x})= \Jac_Nev_{\underline{x}}$. The result follows from Proposition \ref{pos} and from the fact that the normal Jacobian of a map is positive if and only if the map is surjective.
\end{proof}
\begin{oss}\label{zer} For a point $(x_1,\dots,x_k)$ of the diagonal $\Delta\subset\R\Sigma^k$ the map $ev_{x_1,\dots,x_k}$ is not surjective, so that  $\mathcal{D}_d^k=0$ on the diagonal.  
\end{oss}

\section{Proof of the main theorems}\label{pfoo}
This section is devoted to the proof of   Theorem \ref{expon}. For this purpose, we use Markov inequality together with the computations of the central moments of $\# Z_s$ given by Theorem \ref{centermom}.
Theorem \ref{centermom} is itself a consequence of Theorem \ref{equimom}, which computes all the moments of $\# Z_s$. For this, we strongly  use a uniform bound on the $L^{\infty}$-norm of the density function $\mathcal{R}^k_d$. This uniform bound is given by Theorem \ref{boundabove}, that we prove in  Sections \ref{olver}-\ref{snume}. 

\subsection{Uniform bound on the $L^{\infty}$-norm of the density function $\mathcal{R}^k_d$}
The main  result of this section is the following uniform boundedness result for the density function defined in Proposition \ref{denf}. We prove this result in Section \ref{olvtec}.
\begin{thm}\label{boundabove} Let $(\mathcal{L},h)$ be a positive real Hermitian line bundle  over a real Riemann surface $\Sigma$. Then,
there exists a constant $C$ and an integer $d_k\in \mathbb{N}$ such that for any $d\geq d_k$ and $\underline{x}=(x_1,\dots,x_k)\in\R \Sigma^k$
$$\frac{1}{\sqrt{d}^k}\mathcal{R}^k_d(\underline{x})\leq C.$$
\end{thm}
We denote by $\mathbf{d}_h$  the distance on $\Sigma$ induced by the curvature form $\omega$ of  $(\mathcal{L},h)$. The following result is a consequence of the exponential decay of  the Bergman kernel.
\begin{prop}\label{broke2}  Let $m,s\in\mathbb{N}$ be such that $m+s=k$. Let $(x_1,\dots,x_m,y_1,\dots,y_s)\in\R\Sigma^k$ be such that $\mathbf{d}_h(x_i,y_j)>\frac{\log d}{C'\sqrt{d}}$, where $C'$ is the constant appearing in Theorem \ref{fardiag}. Then, we have $$\frac{1}{\sqrt{d}^k}\mathcal{R}_{d}^k(x_1,\dots,x_l,y_1,\dots,y_s)=\frac{1}{\sqrt{d}^l}\mathcal{R}_{d}^l(x_1,\dots,x_l)\frac{1}{\sqrt{d}^s}\mathcal{R}_{d}^s(y_1,\dots,y_s)+O(\frac{1}{d}).$$
\end{prop}
\begin{proof} For clarity of exposition, we prove this lemma for $l=s=1$. The general case follows the same lines. 
Let $x,y\in\R\Sigma^2$ be two points such that $\mathbf{d}_h(x,y)>\frac{\log d}{C'\sqrt{d}}$. The density function write $\mathcal{R}^2_d=\frac{\mathcal{N}_d^2}{\mathcal{D}^2_d}$, see Proposition \ref{df}.
We treat separately the denominator $\mathcal{D}^2_d$ and the numerator $\mathcal{N}_d^2$.\\
 By Theorem \ref{fardiag},  we have $\frac{1}{d}\norm{\mathcal{K}_d(x_i,x_j)}=O(\frac{1}{d})$. By Propositions \ref{df} and \ref{jet}   we then obtain      $$\frac{1}{d}\mathcal{D}^k_d(x,y)=\frac{1}{d}\sqrt{\det\left[ \begin{array}{cc}
 \norm{\mathcal{K}_d(x,x)} & \norm{\mathcal{K}_d(x,y)}\\
 \norm{\mathcal{K}_d(y,x)} & \norm{\mathcal{K}_d(y,y)}\\
 \end{array} \right] }=$$ $$=\displaystyle \frac{1}{d}\sqrt{\norm{\mathcal{K}_d(x,x)}}\sqrt{\norm{\mathcal{K}_d(y,y)}}+O(\frac{1}{d})=\frac{1}{\sqrt{d}}\mathcal{D}_d^1(x)\frac{1}{\sqrt{d}}\mathcal{D}_d^1(y)+O(\frac{1}{d}).$$
 We now study the numerator.
 Let $\R H^0_{x}=\{s\in \R H^0(\Sigma;\mathcal{L}^d), s(x)=0\}$ be the kernel of the evaluation  map at $x$ and   $j^1_{x}:s\in\R H^0_{x} \mapsto \nabla s(x)$ be the $1$-jet evaluation map. The integrand of the numerator depends only on the $1$-jets of sections $s$ at the points $x,y$ so that, after an integration  over $ \ker j^1_{x}\cap \ker j^1_{y}$, we have 
 $$\displaystyle\mathcal{N}^2_d(x,y)=\int_{s\in\big( \ker j^1_{x}\big)^{\perp}\oplus\big( \ker j^1_{y}\big)^{\perp}}\norm{\nabla s(x)}\cdot\norm{\nabla s(y)} d\mu_{\mid \big( \ker j^1_{x}\big)^{\perp}\oplus\big( \ker j^1_{y}\big)^{\perp}}(s).$$
 Let $s_x$ (resp. $s_y$) be a unit norm section generating the orthogonal of $\ker j^1_{x}$ (resp. $\ker j^1_{y}$). We can then write every $s\in\big( \ker j^1_{x}\big)^{\perp}\oplus\big( \ker j^1_{y}\big)^{\perp}$ as $s=as_x+bs_y.$
 As the distance between the points $x$ and $x$ is bigger than $\frac{\log d}{C'\sqrt{d}}$, a result of Tian (see \cite{tian}, Lemma 3.1) implies that the sections $s_x$ and $s_y$ are  asymptotically orthogonal, more precisely $\langle s_x;s_y\rangle_{L^2}=O(\frac{1}{d})$. Moreover the pointwise norm of $\frac{1}{d}\nabla s_x$ at $y$ is $O(\frac{1}{d})$ (and the same holds for $\frac{1}{d}\nabla s_y$ at $x$). These sections are called \emph{peak sections}, see for example \cite{anc,gw2,tian}.
 For these reasons we have that $$\frac{1}{d^2}\mathcal{N}^2_d(x,y)=\int_{(a,b)\in\R^2}\norm{a\frac{1}{d}\nabla s_x(x)+b\frac{1}{d}\nabla s_y(x)}\cdot\norm{a\frac{1}{d}\nabla s_x(y)+b\frac{1}{d}\nabla s_y(y)} \frac{e^{-a^2-b^2}}{\pi} dadb+O(\frac{1}{d})=$$ 
 $$\int_{(a,b)\in\R^2}\norm{a\frac{1}{d}\nabla s_x(x)}\cdot\norm{b\frac{1}{d}\nabla s_y(y)} \frac{e^{-a^2-b^2}}{\pi} dadb+O(\frac{1}{d})=$$
 $$\frac{1}{d}\int_{a\in\R}\norm{a\nabla s_x(x)} \frac{e^{-a^2}}{\sqrt{\pi}} da\cdot\frac{1}{d}\int_{b\in\R}\norm{b\nabla s_x(x)} \frac{e^{-b^2}}{\sqrt{\pi}} db+O(\frac{1}{d}).$$ 
 The last quantity  equals $\frac{1}{d}\mathcal{N}_d^1(x)\frac{1}{d}\mathcal{N}_d^1(y)+O(\frac{1}{d}).$\\
 We then have $$\frac{1}{d}\mathcal{R}^2_d(x,y)=\frac{1}{d}\frac{\mathcal{N}_d^2(x,y)}{\mathcal{D}^2_d(x,y)}=\frac{1}{d}\frac{\mathcal{N}_d^1(x)\mathcal{N}_d^1(y)}{\mathcal{D}^1_d(x)\mathcal{D}^1_d(y)}+O(\frac{1}{d})=\frac{1}{\sqrt{d}}\mathcal{R}^1_d(x)\frac{1}{\sqrt{d}}\mathcal{R}^1_d(y)+O(\frac{1}{d})$$
\end{proof} 
Let $\Delta_d$ be the neighborhood of the diagonal defined by $$\Delta_d\doteqdot \{(x_1,\dots,x_k)\in \R\Sigma^k \mid\hspace{1mm} \exists   \hspace{1mm} i\neq j \hspace{1mm},\mathbf{d}_h(x_i,x_j)< \frac{\log d}{C'\sqrt{d}}\}$$
where $C'$ is the constant appearing in Theorem \ref{fardiag}. 
The following result gives the estimate of the density function for points outside $\Delta_d$.
 \begin{lem}\label{sim2}  Under the hypothesis of Theorem \ref{asymom}, for every $f\in C^0(\R\Sigma^k)$, the following asymptotics hold:
  $$\int_{\R\Sigma^k\setminus\Delta_d}f\frac{1}{\sqrt{d}^k}\mathcal{R}^k_d|\dV_h|^k=\frac{1}{\sqrt{\pi}^k}\int_{\R\Sigma^k\setminus\Delta_d}f|\dV_h|^k+\norm{f}_{\infty}O(\frac{1}{d})$$
  where the error $O(\frac{1}{d})$ does not depend on $f$.
 \end{lem}
  \begin{proof}  
 By Proposition \ref{broke2}, we have, for any $(x_1,\dots,x_k)\in\R\Sigma^k\setminus\Delta_d$,  $$\frac{1}{\sqrt{d}^k}\mathcal{R}_d^k(x_1,\dots,x_k)=\prod_{i=1}^k\frac{1}{\sqrt{d}}\mathcal{R}_d^1(x_i)+O(\frac{1}{d})$$
By  Theorem \ref{esper} in Section \ref{soffest} (see also \cite[Theorem 1.1]{gw2}) we know that $\frac{1}{\sqrt{d}}\mathcal{R}_d^1(x_i)=\frac{1}{\sqrt{\pi}}+O(\frac{1}{d})$, so we have the result. 
\end{proof}
We can prove the first claim of  Theorem \ref{equimom}, that is:
\begin{thm}\label{weaker} Let $(\mathcal{L},h)$ be a positive real Hermitian line bundle  over a real Riemann surface $\Sigma$. For every $f:\R\Sigma^k\rightarrow\R$ bounded  function, we have 
$$\frac{1}{\sqrt{d}^k}\mathbb{E}[\nu_s^k](f)= \frac{1}{\sqrt{\pi}^k}\int_{\R\Sigma^k}f|\dV_h|^k+\norm{f}_{\infty}O(\frac{\log d}{\sqrt{d}}).$$
\end{thm}

\begin{proof} We denote by $\Delta_d$  the neighborhood of the diagonal defined by $$\Delta_d\doteqdot \{(x_1,\dots,x_k)\in \R\Sigma^k \mid\hspace{1mm} \exists   \hspace{1mm} i\neq j \hspace{1mm},\mathbf{d}_h(x_i,x_j)< \frac{\log d}{C'\sqrt{d}}\}$$
where $C'$ is the constant appearing in Theorem \ref{fardiag}. The volume of $\Delta_d$ is $O(\frac{\log d}{\sqrt{d}})$.\\
Thanks to Proposition \ref{denf}, we have $$\frac{1}{\sqrt{d}^k}\mathbb{E}[\tilde{\nu}_s^k](f)=\frac{1}{\sqrt{d}^k}\int_{\underline{x}\in\R \Sigma^k}f(\underline{x})\mathcal{R}_d(\underline{x})|\dV_h|^k$$
where $\tilde{\nu}_s^k$ is defined in Definition \ref{modmom}.
By Lemma \ref{sim2}  we have 
$$\frac{1}{\sqrt{d}^k}\int_{\underline{x}\in\R \Sigma^k\setminus\Delta_d}f(\underline{x})\mathcal{R}_d(\underline{x})|\dV_h|^k= \frac{1}{\sqrt{\pi}^k}\int_{\underline{x}\in\R \Sigma^k\setminus\Delta_d}f(\underline{x})|\dV_h|^k+O(\frac{1}{d}).$$
Moreover, by Theorem \ref{boundabove}, the function  $\frac{1}{\sqrt{d}^k}\mathcal{R}^k_d$ is uniformly bounded by a constant $C$, so that $$\frac{1}{\sqrt{d}^k}\int_{\underline{x}\in\Delta_d}f(\underline{x})\mathcal{R}^k_d(\underline{x}) |\dV_h|^k\leq C\norm{f}_{\infty}\Vol(\Delta_d)=\norm{f}_{\infty}O(\frac{\log d}{\sqrt{d}}).$$
Putting together the integral over $\R\Sigma^k\setminus\Delta_d$ and over $\Delta_d$, we obtain $$\frac{1}{\sqrt{d}^k}\mathbb{E}[\tilde{\nu}_s^k](f)=\frac{1}{\sqrt{\pi}^k}\int_{\underline{x}\in\R \Sigma^k}f(\underline{x})|\dV_h|^k+\norm{f}O(\frac{\log d}{\sqrt{d}}).$$
To conclude we must compare the moment $\mathbb{E}[\nu_s^k]$ with the modified moment $\mathbb{E}[\tilde{\nu}_s^k]$.
 By Proposition \ref{difference}, the quantity $\mathbb{E}[\nu_s^k]-\mathbb{E}[\tilde{\nu}_s^k]$ is a sum of terms of the form $\mathbb{E}[\tilde{\nu}_s^m]$ with $m<k.$ This implies that 
 $(\mathbb{E}[\nu_s^k]-\mathbb{E}[\tilde{\nu}_s^k])(f)=O(\sqrt{d}^{k-1})$ and then  $$\frac{1}{\sqrt{d}^k}\mathbb{E}[\nu_s^k](f)=\frac{1}{\sqrt{\pi}^k}\int_{\underline{x}\in\R \Sigma^k}f(\underline{x})|\dV_h|^k+\norm{f}O(\frac{\log d}{\sqrt{d}}).$$
\end{proof}

\subsection{Proof of Theorems \ref{expon},  \ref{equimom} and \ref{centermom}}
 This section is devoted to the proof of the main theorems. We follow the notation of Section \ref{ransec}. Remember that $\mathcal{R}_d^k$ is the density function given by Proposition \ref{denf}. 
 We need the following
\begin{defn}\label{open1} We denote by $\mathbf{d}_h$ the distance on $\R\Sigma^k$ induced by $h$.\begin{itemize}
\item Let $U_d$ be the union $\displaystyle \cup_{a< b} U_d^{a,b}$, where $U_d^{a,b}$ is the  set
$$\displaystyle U_d^{a,b}\doteqdot \{(x_1,\dots,x_k)\in \R\Sigma^k \mid  \mathbf{d}_h(x_i,x_j)\leq \frac{\log d}{C'\sqrt{d}}\hspace{2mm} \Rightarrow \hspace{1mm} (i,j)=(a,b)\}$$
with $C'$ being the constant appearing in Theorem \ref{fardiag}. \\
Equivalently, a point $(x_1,\dots,x_k)$ lies in $U_d^{a,b}$ if and only if the distance between every pair of points $(x_i,x_j)$, $i\neq j\in\{1,\dots,k\}$, is bigger than $\frac{\log d}{C'\sqrt{d}}$, except at most the pair $(x_a,x_b)$.
\item For any $1\leq a<b\leq k$, we denote by $j_{ab}:\R\Sigma^{k-1}\hookrightarrow \R \Sigma^k$  the inclusion $$(x_1,\dots,x_a,\dots,\hat{x}_b,\dots,x_k)\mapsto (x_1,\dots,x_a,\dots,x_a,\dots,x_k),$$ so that the image $j_{ab}(\R\Sigma^{k-1})$ is equal to $$\{(x_1,\dots.,x_k)\in \R \Sigma^k\mid x_a=x_b\}.$$ 
\end{itemize}
\end{defn}
\begin{lem}\label{volume} Let $U_d$ be the set defined in Definition \ref{open1}. Then $$\Vol_h(\R\Sigma^k\setminus U_d)= O(\frac{(\log d)^2}{d}).$$
\end{lem}
\begin{proof} By definition, $U_d$ is the set of points $(x_1,...,x_k)\in\R\Sigma^k$ such that the distance between each pair of points $x_i,x_j$ , $1\leq i<j\leq k$, is bigger than $\frac{\log d}{C' \sqrt{d}}$, except at most one pair.\\
Its complement $\R\Sigma^k\setminus U_d$ is then formed by points $(x_1,...,x_k)\in\R\Sigma^k$ such that there exist four indices $i,t,s,j\in\{1,\dots,k\}$, $i\neq t$, $i\neq j$, $s\neq j$, with the property that $\mathbf{d}_h(x_i,x_t)\leq \frac{\log d}{C' \sqrt{d}}$ and $\mathbf{d}_h(x_s,x_j)\leq \frac{\log d}{C' \sqrt{d}}$. (Note that the indices $t$ and $s$ could be equal).\\
Fix  such four indices $i,t,s,j$,  and denote by $V_d^{i,t,s,j}$ the set of points $(x_1,...,x_k)\in\R\Sigma^k$ such that $\mathbf{d}_h(x_i,x_t)\leq \frac{\log d}{C' \sqrt{d}}$ and $\mathbf{d}_h(x_s,x_j)\leq \frac{\log d}{C' \sqrt{d}}$. The volume of $V_d^{i,t,s,j}$ is then  $O(\frac{(\log d)^2}{d})$. We then obtain $$\Vol(\R\Sigma^k\setminus U_d)\leq\displaystyle\sum_{i,t,s,j}\Vol(V_d^{i,t,s,j})=O(\frac{(\log d)^2}{d}).$$
\end{proof}
We formulate an uniform estimate of the density function $\mathcal{R}_d^k$ over $U_d$. This estimate, together with Theorem \ref{boundabove}, will imply Theorem \ref{equimom}.
 \begin{prop}\label{simple} Under the hypothesis of Theorem \ref{asymom} and using the notations of Definition \ref{open1}, there exists an universal constant $M'$ such that, for every $f\in C^0(\R\Sigma^k)$, the following asymptotics hold:
$$\frac{1}{\sqrt{d}^k}\int_{U_d}f\mathcal{R}^k_d|\dV_h|^k=\frac{1}{\sqrt{\pi}^k}\int_{U_d}f|\dV_h|^k+
\frac{1}{\sqrt{\pi}^{k-2}}\sum_{a< b}\frac{M'}{\sqrt{d}}\int_{\R\Sigma^{k-1}}j_{ab}^*f|\dV_h|^{k-1}+o\big(\frac{1}{\sqrt{d}}\big) 
 $$
The error term $o(\frac{1}{\sqrt{d}})$ is bounded from above by $$\norm{f}_{\infty}\big(O\big(\frac{1}{\sqrt{d}^{1+\alpha}}\big)+\omega_f\big(\frac{1}{\sqrt{d}^{\alpha}}\big)O(\frac{1}{\sqrt{d}})\big)$$
for any $\alpha\in(0,1)$, where $\omega_f(\cdot)$ is the modulus of continuity of $f$. 
The errors $O\big(\frac{1}{\sqrt{d}^{1+\alpha}}\big)$ and $O(\frac{1}{\sqrt{d}})$ do not depend on $f$. 
 Moreover $M\doteqdot M'+\frac{1}{\sqrt{\pi}}$ is positive.
 \end{prop}
We will prove this Proposition in Section \ref{soffest}. We now deduce Theorem \ref{equimom}.
\begin{proof}[\textbf{Proof of Theorem \ref{equimom}}]
We have to prove the second claim of the theorem, that is the better asymptotic estimate in the case of $f$ continuous function. The first claim was proved in Theorem \ref{weaker}.
 We start by computing the modified moment $\mathbb{E}[\tilde{\nu}_s^k]$, see Definition \ref{modmom}.
Thanks to Proposition \ref{denf}, we have $$\frac{1}{\sqrt{d}^k}\mathbb{E}[\tilde{\nu}_s^k](f)=\frac{1}{\sqrt{d}^k}\int_{\R \Sigma^k}f(\underline{x})\mathcal{R}^k_d(\underline{x})|\dV_h|^k.$$
We divide the integration domain into two parts, namely the set $U_d$ defined in Definition \ref{open1} and its complement $\R\Sigma^k\setminus U_d$.
\begin{itemize}
\item Proposition \ref{simple} implies that $\frac{1}{\sqrt{d}^k}\int_{U_d}f(\underline{x})\mathcal{R}^k_d(\underline{x})|\dV_h|^k$ is equal to 
 $$\frac{1}{\sqrt{\pi}^k}\int_{U_d}f(\underline{x})|\dV_h|^{k}+\frac{1}{\sqrt{\pi}^{k-2}}\sum_{a< b}\frac{M'}{\sqrt{d}}\int_{\R\Sigma^{k-1}}j_{ab}^*f|\dV_h|^{k-1}+o\big(\frac{1}{\sqrt{d}}\big)$$
where the error term $o\big(\frac{1}{\sqrt{d}}\big)$  and the universal constant $M'$ are as in Proposition \ref{simple}.
\item Theorem \ref{boundabove} and Lemma \ref{volume} imply $$\frac{1}{\sqrt{d}^k}\int_{\R \Sigma^k\setminus U_d}f(\underline{x})\mathcal{R}_d^k(\underline{x}) |\dV_h|^k\leq\Vol(\R \Sigma^k\setminus U_d)\norm{f}_{\infty}O(1)=\norm{f}_{\infty}O(\frac{(\log d)^2}{d})$$
where the error $O(\frac{(\log d)^2}{d})$ does not depend on $f$. 
\end{itemize} 
Remark that the integral over $\R \Sigma^k\setminus U_d$ is then negligible compared to the integral over its complement $U_d$. This implies that the integral over  the cartesian product $\R\Sigma^k$ is equal to the integral over  $U_d$, up to an error term. In other words,
 putting together the integral over $\R \Sigma^k\setminus U_d$ and over $U_d$, we obtain   
 \begin{equation}\label{cal} \frac{1}{\sqrt{d}^k}\mathbb{E}[\tilde{\nu}_s^k](f)=\frac{1}{\sqrt{\pi}^k}\int_{\R \Sigma^k}f(\underline{x})|\dV_h|^k+\frac{1}{\sqrt{\pi}^{k-2}}\sum_{a< b}\frac{M'}{\sqrt{d}}\int_{\R\Sigma^{k-1}}j^*_{ab}f|\dV_h|^{k-1}+o\big(\frac{1}{\sqrt{d}}\big)
 \end{equation}
 where the error $o\big(\frac{1}{\sqrt{d}}\big)$ is as in Proposition \ref{simple}. \\
 To conclude we must compare the moment $\mathbb{E}[\nu_s^k]$ with the modified moment $\mathbb{E}[\tilde{\nu}_s^k]$.  By Proposition \ref{difference} we have
 $$\frac{1}{\sqrt{d}^k}(\nu_s^k(f)-\tilde{\nu}_s^k(f))=\frac{1}{\sqrt{d}^k}\sum_{I\in \widetilde{\mathcal{P}}_k}\tilde{\nu}_s^{m_I}(j^*_If)$$
see Section \ref{modmoment} for definition of $\widetilde{\mathcal{P}}_k$ and of $j_{I}$.
By Eq. (\ref{cal}) applied  to $m_I\leq k-2$, we get
$$\frac{1}{\sqrt{d}^k}\mathbb{E}[\tilde{\nu}_s^{m_I}](j^*_If)=\norm{f}_{\infty}O(\frac{1}{\sqrt{d}^{k-m_I}}).$$
We then obtain  $$\frac{1}{\sqrt{d}^k}(\mathbb{E}[\nu_s^k]-\mathbb{E}[\tilde{\nu}_s^k])(f)=\frac{1}{\sqrt{d}^k}\sum_{1\leq a< b\leq k}\mathbb{E}[\tilde{\nu}_s^{k-1}](j_{ab}^*f)+\norm{f}_{\infty}O(\frac{1}{d}).$$ We can  use Eq. (\ref{cal}),  for the modified moments $\mathbb{E}[\tilde{\nu}_s^{k-1}](j_{ab}^*f)$ and we obtain:
 $$\frac{1}{\sqrt{d}^k}(\mathbb{E}[\nu_s^k]-\mathbb{E}[\tilde{\nu}_s^k])(f)=\frac{1}{\sqrt{\pi}^{k-1}\sqrt{d}}\sum_{1\leq a< b\leq k}\int_{\R\Sigma^{k-1}}j_{ab}^*f|\dV_h|^{k-1}+\norm{f}_{\infty}O(\frac{1}{d}).$$
Putting $\frac{1}{\sqrt{d}^k}\mathbb{E}[\tilde{\nu}_s^k](f)$ on the right hand side and using again Eq. (\ref{cal}) we obtain
 $$\displaystyle\frac{1}{\sqrt{d}^k}\mathbb{E}[\nu_s^k](f)=\frac{1}{\sqrt{\pi}^k}\int_{\R \Sigma}f|\dV_h|^k+\frac{M}{\sqrt{\pi}^{k-2}\sqrt{d}}\sum_{a< b}\int_{j_{ab}(\R\Sigma^{k-1})}f_{\mid \{x_a=x_b\}}|\dV_h|^{k-1}+o\big(\frac{1}{\sqrt{d}}\big)$$ 
where $M=M'+\frac{1}{\sqrt{\pi}}$ and where the error term is as in Proposition \ref{simple}. 
 The constant $M$ is positive thanks to Proposition \ref{simple}.
\end{proof}

\begin{proof}[\textbf{Proof of Theorem \ref{centermom}}] We prove the theorem in the case of $f:\R\Sigma\rightarrow\R$ continuous. The case of $f$ bounded  follows the same lines.
Let $f$ be a continuous function on $\R\Sigma$. It induces a continuous function on $\R\Sigma^l$ for every $l\in\mathbb{N}$ defined by $$(x_1,\dots,x_l)\in\R\Sigma^l\mapsto f(x_1)\cdots f(x_l)\in\R.$$ With a slight abuse of notation we still denote this function by $f$.
For any $l\in\mathbb{N}$, we  define $$\mathcal{E}_l(f)\doteqdot\frac{1}{\sqrt{d}^{l-1}}(\mathbb{E}[\nu_s^l](f)-\mathbb{E}[\nu_s]^l(f)).$$
By Theorem \ref{esper}, we have  $$\frac{1}{\sqrt{d}}\mathbb{E}[\nu_s](f)=\frac{1}{\sqrt{\pi}}\int_{\R\Sigma}f|\dV_h|+O\big(\frac{1}{d}\big)$$ 
so that
\begin{equation}\label{defrast}
\frac{1}{\sqrt{d}^l}\mathbb{E}[\nu_s]^l(f)=\frac{1}{\sqrt{\pi}^l}(\int_{\R\Sigma}f|\dV_h|)^l+O\big(\frac{1}{d}\big)
\end{equation} 
for any $l\in\mathbb{N}$.
Combining this with  Theorem \ref{equimom}, we have that for any $l$,
 \begin{equation}\label{defra}\mathcal{E}_l(f)=\frac{M}{\sqrt{\pi}^{l-2}}\sum_{1\leq a<b\leq l}\int_{\R\Sigma^{l-1} }j^*_{ab}f |\dV_h|^{l-1}+o(1)=\frac{Ml(l-1)}{2\sqrt{\pi}^{l-2}}(\int_{\R\Sigma }f |\dV_h|)^{l-1}+o(1)
\end{equation}
where the error term $o(1)$ is bounded from above by $$\norm{f}_{\infty}\big(O\big(\frac{1}{\sqrt{d}^{\alpha}}\big)+\omega_f\big(\frac{1}{\sqrt{d}^{\alpha}}\big)O(1)\big)$$
for any $\alpha\in(0,1)$, where $\omega_f(\cdot)$ is the modulus of continuity of $f$. 
Moreover the errors $O\big(\frac{1}{\sqrt{d}^{\alpha}}\big)$ and $O(1)$ do not depend on $f$.\\
We then have $$\frac{1}{\sqrt{d}^{k-1}}\mathbb{E}\big[(\nu_s-\mathbb{E}[\nu_s])^k\big](f)=
\mathbb{E}\big[\sum_{l=0}^k(-1)^{k-l}\binom{k}{l}\frac{1}{\sqrt{d}^{l-1}}\nu^l_k(f)\frac{1}{\sqrt{d}^{k-l}}\mathbb{E}\big[\nu_s(f)\big]^{k-l}\big]$$
$$=\sum_{l=0}^k(-1)^{k-l}\binom{k}{l}\frac{1}{\sqrt{d}^{l-1}}\mathbb{E}[\nu_s^l(f)]\frac{1}{\sqrt{d}^{k-l}}\mathbb{E}[\nu_s(f)]^{k-l}$$
$$=\sum_{l=0}^k(-1)^{k-l}\binom{k}{l}\big(\frac{1}{\sqrt{d}^{l-1}}\mathbb{E}[\nu_s(f)]^{l}+\mathcal{E}_l(f)\big)\frac{1}{\sqrt{d}^{k-l}}\mathbb{E}[\nu_s(f)]^{k-l}$$
\begin{equation}\label{2}=\sum_{l=0}^k(-1)^{k-l}\binom{k}{l}\big(\frac{1}{\sqrt{d}^{k-1}}\mathbb{E}[\nu_s(f)]^{k}+\mathcal{E}_l(f)\frac{1}{\sqrt{d}^{k-l}}\mathbb{E}[\nu_s(f)]^{k-l}\big)
\end{equation}
By the formula $\sum_{l=0}^k(-1)^l\binom{k}{l}=0$ we have that Eq. (\ref{2})  equals
\begin{equation}\label{3}\sum_{l=0}^k(-1)^{k-l}\binom{k}{l}\mathcal{E}_l(f)\frac{1}{\sqrt{d}^{k-l}}\mathbb{E}\big[\nu_s(f)\big]^{k-l}.
\end{equation}
Now  we use Eq. (\ref{defra}) for the term $\mathcal{E}_l(f)$ and Eq. (\ref{defrast}) for the term $\frac{1}{\sqrt{d}^{k-l}}\mathbb{E}\big[\nu_s(f)\big]^{k-l}$ and we obtain that Eq. (\ref{3})  is equal to
$$\sum_{l=0}^k(-1)^{k-l}\binom{k}{l}l(l-1)\frac{M}{2\sqrt{\pi}^{k-2}}(\int_{\R\Sigma }f |\dV_h|)^{k-1}+o(1).$$
Now, $\sum_{l=0}^k(-1)^{k-l}\binom{k}{l}l(l-1)$ equals $\frac{d^2}{dx^2}_{\mid x=1}(x-1)^k$ and this vanishes as $k>2$.
\end{proof}
\begin{proof}[\textbf{Proof of Theorem \ref{expon}}]
Let $A$ be a measurable subset of $\R\Sigma$ of positive volume.
We have $$\big|\# (Z_s\cap A)-\mathbb{E}[\#(Z_s\cap A)]\big|>c(d)\sqrt{d} \Leftrightarrow \big|\# (Z_s\cap A)-\mathbb{E}[\#(Z_s\cap A)]\big|^k>c(d)^k\sqrt{d}^k$$
so that by Markov's inequality we obtain
$$\mu\big\{\big|\# (Z_s\cap A)-\mathbb{E}[\#(Z_s\cap A)]\big|>c(d)\sqrt{d}\big\}\leqslant \frac{\mathbb{E}[|\#(Z_s\cap A)-\mathbb{E}[\#(Z_s\cap A)]|^k]}{c(d)^k\sqrt{d}^k}$$
By Theorem \ref{centermom} respectively for $f=\mathds{1}_A$ and  $f=1$, we have the result.
\end{proof}
\section[Olver multispaces]{Olver multispaces and proof of Theorem \ref{boundabove}}\label{olvtec}

This section is devoted to the proof of the boundedness result  used in the proof of Theorem \ref{equimom}, namely Theorem \ref{boundabove}.  Olver multispaces and divided differences coordinates play an important role. We recall these tools in Section \ref{olver}.\\ 
The main difficulty in Theorem \ref{boundabove} is to understand the behaviour of the density function $\mathcal{R}_d^k$  near the diagonal, see Remark \ref{zer}. For this purpose, we will study all the possible ways  a point $(x_1,\dots,x_k)\in\R\Sigma^k\setminus\Delta$ can converge to the diagonal. This leads us to consider labelled graph with $k$ vertices. Each vertex represents a point $x_i\in\R\Sigma$ and we put an edge between two vertex if and only if we allow the  points associated with these vertices to collapse each other, see Definition \ref{graph}.
\subsection{Olver multispaces}\label{olver}
We briefly recall the results of \cite{ol}.
Let $M$ an $n$-dimensional smooth (not necessarly connected or closed) manifold  and $C$ and $C'$ two curves in $M$ passing at $z\in M$.
We introduce local coordinates $(x,u_1,\dots,u_{n-1})=(x,u)$ around $z$ such that, locally, the curve $C$ is a graph $\{u=f(x)\}$ and the curve $C'$ a graph $\{u=g(x)\}$, for  smooth maps $f,g$. 

\begin{defn} Let $z=(x_0,u_0)$ be a point in a  smooth manifold  $M$ and $C=\{u=f(x)\}$, $C'=\{u=g(x)\}$ be two curves passing at $z$. We say that $C$ and $C'$ have the same $k$-jet  at $z$ if and only if $f^{(i)}(x_0)=g^{(i)}(x_0)$ for every $i\in\{0,\dots,k\}$. We then write $j_kC\mid_{z}=j_kC\mid_{z}'$.
\end{defn}
\subsubsection{Pointed manifolds}
 A $(k+1)$-\emph{pointed manifold} is an object $\textbf{M}=(M;z_0,\dots,z_k)$ consisting of a smooth manifold $M$ and $(k+1)$ not necessarily distinct points  $z_0,\dots,z_k\in M$ thereon. 
Given $\textbf{M}$, we let $\#i=\#\{j\mid z_j=z_i\}$ denote the number of points which coincide with the $i$-th one.\\
\subsubsection{Multi-contact} 
Given a manifold $M$, we let $\mathcal{C}^{(k)}=\mathcal{C}^{(k)}(M)$ denote the set of all $(k+1)$-pointed curves contained in $X$. In \cite{ol}, Olver defined an equivalence relation on the space of multi-pointed curves that generalizes the jet equivalence relation  at a single point.
\begin{defn} Two $(k+1)$-pointed curves 
$$\textbf{C}=(C;z_0,\dots,z_k)\qquad \widetilde{\textbf{C}}=(\tilde{C};\tilde{z}_0,\dots,\tilde{z}_k)$$
have $k$-order multi-contact  if and only if 
$$z_i=\tilde{z}_i \quad\textrm{and}\quad j_{\#i-1}C\mid_{z_i}=j_{\#i-1}\widetilde{C}\mid_{z_i}$$
for each $i=0,\dots,k$. The $k$-order multi-space, denoted by $M^{(k)}$, is the set of equivalence classes of $(k+1)$-pointed curves in $M$ under the equivalence relation of $k$-th order multi-contact. The equivalence class of an $(k+1)$-pointed curve \textbf{C} is called its $k$-th order multi-jet, and denoted $j_k\textbf{C}\in M^{(k)}$.
\end{defn}
In \cite{ol} Olver proved the following theorem:
\begin{thm}\label{multov}
If $M$ is a smooth manifold of dimension $n$, then its $k$-th order multispace $M^{(k)}$ is a smooth manifold of dimension $(k+1)n$, which contains the off diagonal part $M^{\diamond(k+1)}\doteqdot\{(z_0,\dots,z_k)\in M^{k+1}\mid z_i\neq z_j \hspace{1mm}\forall\hspace{1mm} i\neq j\}$ of the Cartesian product space as an open, dense subset, and the $k$-order jet space $J^k$ as a smooth submanifold.
\end{thm}
\begin{oss} For a one-dimensional variety, we have $M^{(k)}\simeq M^{k+1}$ where the canonical isomorphism is given by the forgetful map $[C;z_0,\dots.,z_k]\mapsto (z_0,\dots,z_k)$ with the inverse given by $(z_0,\dots,z_k)\mapsto [M;z_0,\dots,z_k]$.
\end{oss}
\subsubsection{Divided differences} 
Olver proved  Theorem \ref{multov} by providing an atlas. This atlas is formed by local charts given by divided differences.
\begin{defn}\label{divided}   Let $(C;z_0,\dots,z_k)$ be a  $(k+1)$-pointed curve in a smooth manifold of dimension $n$. We introduce local coordinates $(x,u_1,\dots,u_{n-1})=(x,u)$  and we suppose that the curve $C$ is the graph $\{u=f(x)\}$ of a smooth function $f$, so that $z_i=(x_i,f(x_i))$.
We define the divided differences of $(C;z_0,\dots,z_k)$ recursively by setting $[z_0]_C=f(x_0)$ and 
$$[z_0\dots z_{i-1}z_i]_C=\lim_{x\rightarrow x_i}\frac{[z_0\dots z_{i-2}z]_C-[z_0\dots z_{i-1}]_C}{x-x_{i-1}}$$
for any $i\in\{1,\dots k\}$. 
For example, $[z_0z_1]_C=\lim_{x\rightarrow x_1}\frac{z-z_0}{x-x_0}$.
When taking the limit, the points $z=(x,f(x))$ must lie on the curve $C$.
\end{defn}

Olver proved that two $(k+1)$-pointed curves $(C;z_0,\dots,z_k)$ and $(C';z_0,\dots,z_k)$ represent the same element in $M^{(k)}$ if and only if they have same divided differences, see \cite[Theorem 3.4]{ol}.\\
The following lemma will be useful for the local computations in Sections \ref{sdenom} and \ref{snume}.
\begin{lem}\label{pratique} Let $\underline{T}=(T_1,\dots,T_k)\in\R^k$ and $f\in C^k(\R)$ be a $C^k$-function and $C=\textrm{graph}(f)$. Then 
$$[f(T_1)\dots f(T_k)]_C=\sum_{i=1}^k\sum_{r=0}^{\#T_i-1}c_{\underline{T},i,r}f^{(r)}(T_i)$$
where $c_{\underline{T},i,r}$ is a rational function in the distances $T_s-T_t$ ,$1\leq s < t \leq k$ of non positive degree $r-k+1$.
Here, $f^{(r)}(T)$ is the $r$-th derivative of $f$ at $T$.
\end{lem}
In the sequel we will omit $C$ in the notation of divided differences, that is we write $[\cdots]$ instead of $[\cdots]_C$.
\begin{proof} 
By induction on $k$. If $k=1$ we have $[f(T_1)]=f(T_1)$.\\
By definition, we have $$[f(T_1)\dots f(T_{k-1})f(T_k)]=\lim_{T\rightarrow T_k}\frac{[f(T_1)\dots f(T_{k-2})f(T)]-[f(T_1)\dots f(T_{k-1})]}{T-T_{k-1}}$$
and  by induction hypothesis we have $$[f(T_1)\dots f(T_{k-1})]=\sum_{i=1}^{k-1}\sum_{r=0}^{\#T_i-1}c_{\underline{T'},i,r}f^{(r)}(T_i)$$ with $\underline{T'}=(T_1,\dots,T_{k-1})\in\R^{k-1}.$
If $T_k\neq T_{k-1}$, then 
$$[f(T_1)\dots f(T_{k-1})f(T_k)]=\frac{[f(T_1)\dots f(T_{k-2})f(T_k)]}{{T_k-T_{k-1}}}-\frac{[f(T_1)\dots f(T_{k-1})]}{T_k-T_{k-1}}$$
and by induction we have the result.\\
If $T_k=T_{k-1}$, then $$[f(T_1)\dots f(T_{k-1})f(T_{k})]=\frac{d}{dT_{k-1}}[f(T_1)\dots f(T_{k-1})]=\frac{d}{dT_{k-1}}\sum_{i=1}^{k-1}\sum_{r=0}^{\#T_i-1}c_{\underline{T'},i,r}f^{(r)}(T_i)$$
and, again, by induction hypothesis we have the result.
\end{proof}
\subsection{Extension of the density function $\mathcal{R}^k_d$}\label{42}
In this subsection we introduce an incidence manifold that will play a key role in the proof of Theorem \ref{boundabove} and of the following:

\begin{thm}\label{extdens}
For large $d$, the density function $\mathcal{R}_d^k:\R\Sigma^k\setminus\Delta\rightarrow\R$ in Proposition \ref{denf} can be extended to a smooth function $\R\Sigma^k\rightarrow\R$ that we still denote by $\mathcal{R}_d^k$.  Moreover, $\mathcal{R}_d^k$ vanishes on the diagonal $\Delta$. 
\end{thm}

\subsubsection{Extension of the incidence manifold} Let $\R \mathcal{L}^{d}$ be the real locus of the total space of the line bundle $\mathcal{L}^d\rightarrow\Sigma$. It is an open manifold of real dimension $2$. For $k\in\mathbb{N}$, we  take its multi-space $(\R \mathcal{L}^{d})^{(k-1)}$, see Section \ref{olver}.\\
Consider the map
$$ev:\R H^0(\Sigma;\mathcal{L}^d)\times \R \Sigma^k\rightarrow (\R \mathcal{L}^{d})^{(k-1)}$$
defined by $$(s,x_1,\dots,x_k)\mapsto\big(\textrm{graph}(s);s(x_1),\dots,s(x_k)\big).$$
Let $\mathbf{Z}$ be the set of points in $(\R \mathcal{L}^{d})^{(k-1)}$ which represent the $k$-pointed curves  $(\R\Sigma;x_1,\dots,x_k)$, where $\R\Sigma$ is viewed as the zero section of 
$\R \mathcal{L}^{d}$ and $x_1,\dots,x_k\in\R\Sigma$.
\begin{defn}\label{definc2} We  denote the set $ev^{-1}(\mathbf{Z})$ by $\mathcal{I}$ and call it  the \emph{incidence manifold}. 
\end{defn}
\begin{oss} For any $(x_1,\dots,x_k)\in\R \Sigma^k\setminus \Delta$ we have $(s,x_1,\dots,x_k)\in ev^{-1}(\mathbf{Z})$ if and only if $s(x_i)=0$ for any $i\in\{1,\dots,k\}$. Hence, the Definition \ref{definc2} extends the definition of the incidence manifold given in Section \ref{incsect}. 
\end{oss}

\subsection{Local equations of the incidence manifold}\label{secgraph}
  We give local equations for the incidence manifold $\mathcal{I}$ defined in Definition \ref{definc2}. These local equations allow us to find a new local expression for the density function $\mathcal{R}^k_d$. We will use the divided differences notation  (see Section \ref{olver}, Definition \ref{divided}). 
\subsubsection{A partition of $\R\Sigma^k$}   
Let $\Theta_k$ be the set of all labelled graphs with exactly $k$ vertices, labelled by $\{1,\dots,k\}$. 
 Let $d$ be any positive integer, we will associate a graph in $\Theta_k$ to any point $(x_1,\dots,x_k)\in\R\Sigma^k$. This graph is contructed as follows:  we put an edge between the $i$-th vertex and $j$-th vertex if and only if the distance $\mathbf{d}_h(x_i,x_j)$ between $x_i$ and $x_j$ is smaller or equal than $\frac{1}{\sqrt{d}}$.
 \begin{defn}\label{graph} \begin{itemize}
\item Given an integer $d\in\mathbb{N}$ and a point $\underline{x}\in\R\Sigma^k$ we say that $\underline{x}\in\Gamma_d$ if and only if its associated graph is $\Gamma$. We will refer to $\Gamma_d$ as a \emph{graph subset}.
\item We call  \emph{origin} of a connected component of such a graph the point of the connected component with the smallest label.
\item If $\Gamma$ has $m$ connected components $\Gamma^1,\dots,\Gamma^m$ and each connected component $\Gamma^i$ has $k_i$ vertices, then we write $\{x^i_1,\dots,x^i_{k_i}\}$ for the corresponding points of the connected component. We then have $\{x^i_p\}^{i=1,\dots,m}_{p=1,\dots,k_i}=\{x_1,\dots,x_k\}$. \\
We say that $\{x^i_1,\dots,x^i_{k_i}\}$ is a \emph{connected component} of $\underline{x}$ and that $x^i_p$ is a \emph{vertex} of $\underline{x}$.
\end{itemize}
\end{defn}
\begin{oss}\begin{itemize} \item For any $d\in \mathbb{N}$, the graph subsets $\{\Gamma_d\}_{\Gamma\in\Theta_k}$ give a cover of $\R \Sigma^k$.
\item The fact that two point $x_i$ and $x_j$ lie in the same connected components implies that the distance between these two points is smaller or equal than $\frac{k}{\sqrt{d}}$.
\end{itemize}
\end{oss} 

\subsubsection{Generalized evaluation maps} Fix a $k$-labelled graph $\Gamma\in\Theta_k$ with $m$ connected components. We write $\{x^i_1,\dots,x^i_{k_i}\}_{1\leq i \leq m}$ for the connected components of a point $\underline{x}$ that lies in $\Gamma_d$.
We take real normal coordinates $U_i$ around  $x_1^i$, see Section \ref{normal}. We identify each point $x_p^i$ with the respective coordinate in the chart $U_i$.  We recall that  with  a real normal chart comes together with a real trivialization of $\R\mathcal{L}^d_{\mid U}$, see Section \ref{normal}. 
\begin{defn}\label{eval}
Given a point $\underline{x}\in\Gamma_d$ in normal coordinates around the origin $x_1^i$, we define the \emph{generalized evaluation map} 
$$ev_{\underline{x}}^{\Gamma}:\R H^0(\Sigma;\mathcal{L}^d)\rightarrow \R^{k_1}\times\cdots\times\R^{k_m}$$
by $$s\mapsto ([s(x_1^1)]_{C},\dots,[s(x_1^1)\cdots s(x_{k_1}^1)]_C,\dots,[s(x_1^m)]_C,\dots,[s(x_1^m)\cdots s(x_{k_m}^m)]_C)$$
where $C=\textrm{graph}(s)$.\\
The kernel of such map is denoted by $\R H^0_{\underline{x}}$. 
\end{defn}
Remark that for points $\underline{x}=(x_1,\dots,x_k)\in\R\Sigma^k\setminus\Delta$ outside the diagonal, the kernel of the generalized evaluation map equals the kernel of the "classical" evaluation map $s\mapsto (s(x_1),\dots,s(x_k))$. 
\begin{conv}\label{con1} In the sequel, we will only consider  $k$-pointed curves of the form 
$$\big(\textrm{graph}(s);s(x_1),\dots,s(x_k)\big)$$ for $s\in\R H^0(\Sigma;\mathcal{L}^d)$ and for $x_1,\dots,x_k\in\R\Sigma$. For the sake of simplicity,  we will  omit "$\textrm{graph}(s)$" in the notation of its divided differences. For example, we will write  $[s(x_1)\cdots s(x_i)]$ instead of $[s(x_1)\cdots s(x_i)]_{\textrm{graph}(s)}$.
\end{conv} 
In the following proposition we describe the incidence variety $\mathcal{I}$ defined in Definition \ref{definc2} in terms of divided differences.
\begin{prop}\label{prima} Let $\Gamma$ be a $k$-labelled graph,  $s\in\R H^0(\Sigma;\mathcal{L}^d)$ be a real section and $\underline{x}=(x_1,\dots,x_k)\in\Gamma_d$ be a point in the graph subset $\Gamma_d$. Denote  by $\{x^i_1,\dots,x^i_{k_i}\}$ the  connected components of $\underline{x}$, $i\in\{1,\dots,m\}$. 
Then $(s,\underline{x})\in\mathcal{I}$ if and only if 
$$[s(x^i_1)]=0,[s(x^i_1)s(x^i_2)]=0,\dots,[s(x^i_1)s(x^i_2)\cdots s(x^i_{k_i})]=0$$
for any $i\in\{1,\dots,m\}$, that is if and only if $ev_{\underline{x}}^{\Gamma}(s)=0$ where $ev_{\underline{x}}^{\Gamma}$ is the generalized evaluation map defined in Definition \ref{eval}.
\end{prop}
\begin{proof}  In local coordinates the equivalence relation "having multicontact" $\big(s(\R\Sigma);s(x^i_1),\dots,s(x^i_{k_i})\big)\sim \big(\R\Sigma;x^i_1,\dots,x^i_{k_i}\big)$ reads $[s(x^i_1)]=0,[s(x^i_1)s(x^i_2)]=0,\dots,[s(x^i_1)s(x^i_2)\cdots s(x^i_{k_i})]=0.$
Now, having $k$-th order multi-contact is a local property, and then  $\big(s(\R\Sigma);s(x_1),\dots,s(x_k)\big)\sim \big(\R\Sigma;x_1,\dots,x_k\big)$ if and only if $\big(s(\R\Sigma);s(x^i_1),\dots,s(x^i_{k_i})\big)\sim \big(\R\Sigma;x^i_1,\dots,x^i_{k_i}\big)$ for any $i=1,\dots,m$ .
\end{proof}
These equations using divided differences show that $\mathcal{I}$ is a smooth manifold:
\begin{prop} Let $\mathcal{L}$ be a positive real line bundle over  a real Riemann surface $\Sigma$. Then, for large $d\in\mathbb{N}$, the set $\mathcal{I}$ defined in Def. \ref{definc2} is a smooth manifold. 
Moreover the incidence manifold $\mathcal{I}$ defined in Section \ref{incsect} is an open dense subset of $\mathcal{I}$, which coincide with $\pi_{\Sigma}^{-1}(\R\Sigma\setminus\Delta)$, where $\pi_{\Sigma}:\mathcal{I}\rightarrow \R\Sigma^k$ is the natural projection.
\end{prop}
\begin{proof}
We first make the following remark which is a direct consequence of the positivity of $\mathcal{L}$ and of the Riemann-Roch theorem:
there exists an integer $d_k$ such that, for any $d\geq d_k$, any  $k$-admissible graph $\Gamma$ and  any $\underline{x}=(x_1^1,\dots,{x}^1_{k_1},\dots,{x}^m_1,\dots,{x}^m_{k_m})\in \Gamma_d$,   we can find a section $\dot{s}_p^i\in\R H^0(\Sigma;\mathcal{L}^d)$ such that  $ev_{\underline{x}}^{\Gamma}(\dot{s}_p^i)=(0,\dots,0, 1,0,\dots, 0)$, where $1$ is exactly in the place corresponding to $x^i_p$. 

We now prove that $\mathcal{I}$ is a smooth manifold.
We have seen in Proposition \ref{prima} that $(s,\underline{x})\in \pi^{-1}_{\Sigma}(\Gamma_d)\cap \mathcal{I}$ if and only if $ev_{\underline{x}}^{\Gamma}(s)=0$. Then, we have to prove that $0$ is a regular value of the map $(s,\underline{x})\mapsto  ev_{\underline{x}}^{\Gamma}(s).$ This means that, for any $\underline{\eta}=(\eta^1_1,\dots,\eta^1_{k_1},\dots,\eta^m_1,\dots,\eta^m_{k_m})\in\R^k$, we have to  find a vector $(\dot{s},\dot{\underline{x}})=(\dot{s},\dot{x}^1_1,\dots,\dot{x}^1_{k_1},\dots,\dot{x}^m_1,\dots,\dot{x}^m_{k_m})$  such that
$$ [\dot{s}(x^i_1)]+[s(x^i_1)s(x^i_1)]\dot{x}^i_1=\eta^i_1 $$  
  $$[\dot{s}(x_1^i)\dot{s}(x_2^i)]+[s(x^i_1)s(x^i_2)s(x^i_1)]\dot{x}^i_1+ [s(x^i_1)s(x^i_2)s(x^i_2)]\dot{x}^i_2=\eta^i_2 $$ $$ \vdots $$
   $$ [\dot{s}(x_1^i)\cdots \dot{s}(x_{k_i}^i)]+[s(x^i_1)\cdots s(x^i_{k_i})s(x^i_1)]\dot{x}^i_1+\dots + [s(x^i_1)\cdots s(x^i_{k_i})s(x^i_{k_i})]\dot{x}^i_{k_i}=\eta^i_{k_i} $$
 for any $i\in\{1,\dots,m\}$. \\
  To do this, we   choose $\dot{x}_p^i=0$ for any  $i\in\{1,\dots,m\}$ and any $p\in\{1,\dots,k_i\}$ and $\dot{s}=\sum_{i=1}^m\sum_{p=1}^{k_i}\eta^i_p\dot{s}_p^i.$

\end{proof} 
\subsubsection{New expressions for the density function $\mathcal{R}_d^k$}
\begin{conv} Let $\Gamma$ be a $k$-labelled graph. We denote  the restriction of the density function $\mathcal{R}_d^k$ to the graph subset $\Gamma_d$ by $\mathcal{R}^{\Gamma}_{d}$. 
\end{conv} 
We want to estimate $\mathcal{R}^{\Gamma}_{d}$ for any $k$-labelled graph $\Gamma\in\Theta_k$.
By the finiteness of the set of $k$-labelled graphs $\Gamma$ we obtain that Theorem \ref{boundabove} is equivalent to the following:
\begin{prop}\label{totalbound} Let $\Gamma$ be a $k$-labelled graph and $\mathcal{R}^{\Gamma}_d$ be the restriction of the density function $\mathcal{R}_d^k$ on the graph  subset $\Gamma_d$ defined in Def. \ref{graph}. Then, there exists a constant $C_{\Gamma}$ and a positive integer $d_{\Gamma}$ such that for any 
$d\geq d_{\Gamma}$ and for any $\underline{x}=(x_1,\dots,x_k)\in \Gamma_d$
\[\frac{1}{\sqrt{d}^k}\mathcal{R}^{\Gamma}_d(\underline{x})\leq C_{\Gamma}.\]
\end{prop}
\begin{defn}\label{adm} Let $\Gamma$ be a $k$-labelled graph and let $\Gamma^1,\dots,\Gamma^m$ be its connected components. We denote by $i_1,\dots,i_{k_i}$ the vertices of $\Gamma^i$. We say that $\Gamma$ is an \emph{$k$-admissible graph} if and only if $i_p<j_q$ for any $i<j\in\{1,\dots,m\}$ and any $p\in\{1,\dots,k_i\}$, $q\in\{1,\dots,k_j\}$.
\end{defn}
\begin{oss}\label{aed} By Proposition \ref{symact}  we have $\mathcal{R}^k_d=\mathcal{R}^k_d\cdot \sigma$  for any $\sigma\in\mathcal{S}_k$, so that it suffices to prove Proposition \ref{totalbound} for one particular element of each orbit of the action of the symmetric group $\mathcal{S}_k$ on the $k$-labelled graphs.
As in any orbit of the action of $\mathcal{S}_k$ on the $k$-labelled graphs there is an admissible graph, it suffices to prove  Proposition \ref{totalbound} for $k$-admissible graphs. Moreover, by Proposition \ref{broke2}, it suffices to prove Proposition \ref{totalbound} for points $\underline{x}=(x_1,\dots,x_k)$ such that $\mathbf{d}_h(x_i,x_j)<\frac{\log d}{C'\sqrt{d}}$.
\end{oss}
\begin{conv}  Let $\Gamma$ be a $k$-admissible graph. In the sequel, we will always denote the  points $\underline{x}$ in $\Gamma_d$ by $$\underline{x}=(x_1^1,\dots,x^1_{k_1},\dots,x_1^m,\dots,x^m_{k_m})$$
 where $\big\{x_1^i,\dots,x^i_{k_i}\big\}_i$ are the connected components of $\underline{x}$. As we have seen,  this implies  that $\mathbf{d}_h(x^i_p,x^i_q)\leq \frac{k}{\sqrt{d}}$ and $\mathbf{d}_h(x^i_p,x^j_q)> \frac{1}{\sqrt{d}}$ for $i \neq j$.
\end{conv}
The following proposition gives an explicit expression of  $\mathcal{R}^k_{d\mid_{\Gamma_d}}\doteqdot\mathcal{R}^{\Gamma}_d(\underline{x})$, where $\Gamma$ is a $k$-admissible graph.
We recall that we denote  by $\R H^0_{\underline{x}}$ the kernel of the generalized evaluation map defined in Definition \ref{eval}. 
\begin{prop}\label{frac} Let $\Gamma$ be a $k$-admissible graph.
For any  $\underline{x}\in \Gamma_d$ we have
$\mathcal{R}^{\Gamma}_d(\underline{x})=\frac{\mathcal{N}_d^{\Gamma}(\underline{x})}{\mathcal{D}_d^{\Gamma}(\underline{x})}$,
  where \[\mathcal{N}_d^{\Gamma}(\underline{x})=\int_{\R H^0_{\underline{x}}}
 \prod_{i=1}^{m}|[s(x^i_1)s(x^i_1)]|\cdot|[s(x^i_1)s(x^i_2)s(x^i_2)]|\cdots|[s(x^i_1)\cdots s(x^i_{k_i})s(x^i_{k_i})]|d\mu_{\mid \R H^0_{\underline{x}}}\]
  and \[\mathcal{D}_d^{\Gamma}(\underline{x})=|\Jac_N(ev^{\Gamma}_{\underline{x}})|.\]
\end{prop}
\begin{proof}
The proof goes along  the same lines as Proposition \ref{df}.
Let $\underline{x}\in \Gamma_d$ be any point in $\Gamma_d$. Then we have seen that $(s,\underline{x})\in \pi^{-1}_{\Sigma}(\Gamma_d)\cap \mathcal{I}$ if and only if $ev_{\underline{x}}^{\Gamma}(s)=0$. Differentiating this equation we see that a vector $(\dot{s},\dot{\underline{x}})=(\dot{s},\dot{x}^1_1,\dots,\dot{x}^1_{k_1},\dots,\dot{x}^m_1,\dots,\dot{x}^m_{k_m})$  is in the tangent space of $(s,\underline{x})$ if and only if

  $$ [\dot{s}(x^i_1)]+[s(x^i_1)s(x^i_1)]\dot{x}^i_1=0 $$  
  $$[\dot{s}(x_1^i)\dot{s}(x_2^i)]+[s(x^i_1)s(x^i_2)s(x^i_1)]\dot{x}^i_1+ [s(x^i_1)s(x^i_2)s(x^i_2)]\dot{x}^i_2=0 $$ $$ \vdots $$
   $$ [\dot{s}(x_1^i)\cdots \dot{s}(x_{k_i}^i)]+[s(x^i_1)\cdots s(x^i_{k_i})s(x^i_1)]\dot{x}^i_1+\dots + [s(x^i_1)\cdots s(x^i_{k_i})s(x^i_{k_i})]\dot{x}^i_{k_i}=0 $$
 for $i=1,\dots,m$. 
 This is equal to $ev_{\underline{x}}^{\Gamma}(\dot{s})+D_{\underline{x}}ev_{\underline{x}}^{\Gamma}(s)\dot{\underline{x}}=0$. Here, $D_{\underline{x}}ev^{\Gamma}_{\underline{x}}(s)$ is the derivative of $ev_{\underline{x}}(s)$ with respect to $\underline{x}$ and it is equal to
 the lower triangular matrix that has the following blocks on the diagonal
$$\begin{bmatrix}
 [s(x^i_1)s(x^i_1)] & 0 & \dots & 0 \\
[s(x^i_1)s(x^i_2)s(x^i_1)] & [s(x^i_1)s(x^i_2)s(x^i_2)]&\dots & 0 \\
 \vdots & \vdots & \ddots & \vdots \\
 [s(x^i_1)\cdots s(x^i_{k_i})s(x^i_1)] & [s(x^i_1)\cdots s(x^i_{k_i})s(x^i_2)] & \dots & [s(x^i_1)\cdots s(x^i_{k_i})s(x^i_{k_i})]
 \end{bmatrix}
 $$
  for $i=1,\dots,m$. 
 Writing in a compact notation, we have $$\dot{\underline{x}}=-(D_{\underline{x}}ev^{\Gamma}_{\underline{x}}(s))^{-1}ev_{\underline{x}}^{\Gamma}(\dot{s}).$$
 We then deduce that the normal Jacobian that we want to compute is $$|\Jac_N\pi_{\Sigma}|=\frac{|\Jac_N ev^{\Gamma}|}{|\Jac D_{\underline{x}}ev^{\Gamma}_{\underline{x}}(s)|}.$$
We conclude by observing that $D_{\underline{x}}ev^{\Gamma}_{\underline{x}}(s)$ is triangular, so that its Jacobian is $$\prod_{i=1}^{m}|[s(x^i_1)s(x^i_1)]|\cdot|[s(x^i_1)s(x^i_2)s(x^i_2)]|\cdots|[s(x^i_1)\cdots s(x^i_{k_i})s(x^i_{k_i})]|.$$  
\end{proof}

\subsubsection{Proof of Theorem \ref{boundabove}} 
Proposition \ref{totalbound} (and then Theorem \ref{boundabove}, see Remark \ref{aed}) is a consequence of the following two propositions in which we study separately the numerator $\mathcal{N}_d^{\Gamma}$ and the denominator $\mathcal{D}_d^{\Gamma}$. 
We denote by $C'$ the constant that appeared in Theorem \ref{fardiag}.
\begin{prop}\label{denombound} Let $\Gamma$ be a $k$-admissible graph and $\mathcal{D}^{\Gamma}_d$ be the function defined in Proposition \ref{frac}. Then, there exists a positive $\epsilon_{\Gamma}>0$ and a positive integer $d_{\Gamma}$ such that for any 
$d\geq d_{\Gamma}$ and for any $\underline{x}\in \Gamma_d$ with $\mathbf{d}_h(x^i_p,x^j_a)\leq \frac{\log d}{C'\sqrt{d}}$, we have 
$$\frac{1}{\sqrt{d}^k}\prod_{i=1}^m\frac{1}{\sqrt{d}^{\frac{k_i(k_i-1)}{2}}}\mathcal{D}^{\Gamma}_d(\underline{x}) >\epsilon_{\Gamma}.$$
\end{prop}

\begin{prop}\label{numbound}
Let $\Gamma$ be a $k$-admissible graph and $\mathcal{N}^{\Gamma}_d$ be the function defined in Proposition \ref{frac}. Then, there exists  $M_{\Gamma}$  such that for any 
$d$ and for any $\underline{x}\in \Gamma_d$ with $\mathbf{d}_h(x^i_p,x^j_a)\leq \frac{\log d}{C'\sqrt{d}}$, we have
\[\frac{1}{d^k}\prod_{i=1}^m\frac{1}{\sqrt{d}^{\frac{k_i(k_i-1)}{2}}}\mathcal{N}^{\Gamma}_d(\underline{x})<M_{\Gamma}.\]
\end{prop}
We prove Propositions \ref{denombound} and \ref{numbound} respectively in Sections \ref{sdenom}  and \ref{snume}. We now prove Proposition \ref{totalbound} and Theorem \ref{extdens}.
 \begin{proof}[\textbf{Proof  of Proposition \ref{totalbound}}]
As we have seen,  it suffices to prove this proposition for  $k$-admissible graphs and for points $\underline{x}=(x_1,\dots,x_k)$ such that $\mathbf{d}_h(x_i,x_j)<\frac{\log d}{C'\sqrt{d}}$, see Remark \ref{aed}. The proposition  follows directly from  the fact that  $$\mathcal{R}^{\Gamma}_d(\underline{x})=\frac{\mathcal{N}_d^{\Gamma}(\underline{x})}{\mathcal{D}_d^{\Gamma}(\underline{x})}$$ and from Propositions \ref{numbound} and \ref{denombound}. 
\end{proof}

\begin{proof}[\textbf{Proof of Theorem \ref{extdens}}]
It is a consequence of the surjectivity of $$d\pi_{\Sigma}:T_{(s,\underline{x})}\mathcal{I}\rightarrow T_{\underline{x}}\R\Sigma^k,$$ that is equivalent to the fact that $\mathcal{D}^{\Gamma}_d$ is everywhere non vanishing for any graph $\Gamma\in\Theta_k$. \\
Let us prove that the density function $\mathcal{R}^k_{d}$ vanishes on the diagonal.
Let $(s,\underline{x})\in\mathcal{I}$ be such that $\underline{x}\in\Delta$. Then, up to an action of the symmetric group $\mathcal{S}_k$, we have $\underline{x}=(x_1,x_1,\underline{\tilde{x}})$ for some $\underline{\tilde{x}}\in\R\Sigma^{k-2}$. Then $\mathcal{N}_d^{\Gamma}(\underline{x})=0$ because we integrate $|[s(x_1)s(x_1)]|$ over $\R H^0_{\underline{x}}$ (that is the kernel of the generalized evaluation map defined in Definition \ref{eval}) and in $\R H^0_{\underline{x}}$ we have that $[s(x_1)s(x_1)]=0$.
As $\mathcal{D}^{\Gamma}_d\neq 0$ for any $\underline{x}\in\R \Sigma^k$, we obtain $\mathcal{R}^k_d=0$ on $\Delta$.
\end{proof}

\subsection{Estimates of the denominator $\mathcal{D}_d^{\Gamma}$}\label{sdenom}
This subsection is devoted to the proof of Proposition \ref{denombound}. We use the notations of Sections \ref{olver}--\ref{secgraph}.\\
We introduce some evaluation maps $ev_{\underline{T}}^{\C,\Gamma}$ on the real Bargmann-Fock space $\R H_{L^2}^0(\C,\mathcal{O})$.  These maps turn out to be the local models of the generalized evaluation maps $ev^{\Gamma}_{\underline{x}}$ defined in  Definition \ref{eval}. For this reason, we call them \emph{local evaluation maps}.  The surjectivity of the local evaluation maps on the Bargmann-Fock space will then imply the surjectivity of the generalized evaluation maps $ev^{\Gamma}_{\underline{x}}$, proving Proposition \ref{denombound}.

\subsubsection{A partition of $\R^k$} We introduce a partition of $\R^k$ indexed by labelled graphs, as we did for $\R\Sigma^k$  in Section \ref{secgraph}. Denote by
 $\Theta_k$  the set of all labelled graphs with $k$ vertices, labelled by $\{1,\dots,k\}$. 
We associate a graph to every point $(T_1,\dots,T_k)\in  \R^k$. 
This  is constructed as follows:  we put an edge between the $i$-vertex and $j$-vertex if and only if the distance between $T_i$ and $T_j$ is smaller or equal than $1$.
We say that $(T_1,\dots,T_k)\in\Gamma$ if and only if its associated graph is $\Gamma$. The subset $\Gamma\subset \R^k$ is called a \emph{graph subset}.
\begin{defn} Let $(T_1,\dots,T_k)$ be a point in $\R^k$ and $\Gamma$ be its associated graph. 
\begin{itemize} 
\item We say that $T_i$ and $T_j$ are in the same connected component if and only if the vertices $i$ and $j$ of the associated graph are in the same connected component.
\item We call  \emph{origin} of a connected component the point of the connected component with the smallest label.
\end{itemize} 
\end{defn} 
\subsubsection{A local evaluation model}\label{local2}
\paragraph*{Real Bargmann-Fock space} Let $\R H^0_{L^2}(\C;\mathcal{O})$ be  the space of  functions of the form $f(z)e^{-\frac{|z|^2}{2}}$, where $f:\C\rightarrow\C$ is an holomorphic real function (that is  $f(\bar{z})=\bar{f(z)}$)  such that $\int_{\C}|f|^2e^{-|z|^2}dzd\bar{z}<\infty$. 
We call this space the Bargmann-Fock space. It is naturally equipped with the $L^2$-scalar product $$\langle fe^{-\frac{|z|^2}{2}},ge^{-\frac{|z|^2}{2}} \rangle=\int_{\C}f\bar{g}e^{-|z|^2}dzd\bar{z}.$$ 
The Bergman kernel of the  orthogonal projection $L^2(\C)\rightarrow \R H^0_{L^2}(\C;\mathcal{O})$ equals the local Bergman kernel $K_{\C}(x,y)=\frac{1}{\pi}e^{\frac{-\norm{x-y}^2}{2}}$ for any $x,y\in \C$ (see Definition \ref{scalbegmkern}). 
 An orthonormal basis of $\R H^0_{L^2}(\C;\mathcal{O})$ is $\{\frac{z^k}{\sqrt{\pi}\sqrt{k!}}e^{\frac{-|z|^2}{2}}\}$.
\paragraph*{Local evaluation maps}
In the sequel we will only consider $k$-admissible graphs, see Definition \ref{adm}.
For every $k$-admissible graph $\Gamma$ and   $\underline{T}=(T_1^1,\dots,T^1_{k_1},\dots,T_1^m,\dots,T^m_{k_m})\in \Gamma$ we consider the following evaluation map:
$$ev^{\C,\Gamma}_{\underline{T}}:\R H^0_{L^2}(\C;\mathcal{O})\rightarrow \R^k$$
defined by $$f\mapsto ([f(T_1^1)],\dots,[f(T_1^1)\cdots f(T_{k_1}^1)],\dots,[f(T_1^m)],\dots,[f(T_1^m)\cdots f(T_{k_1}^m)]).$$
The notation $[f(T_1^i)\dots f(T_{p}^i)]$ stands for the divided difference $[f(T_1^i)\dots f(T_{p}^i)]_{\textrm{graph}(f)}$ in the sense of Olver, see Definition \ref{divided}. 
Remark that we used the same notation in the previous section, see Notation \ref{con1}.
\begin{defn}\label{local}
We call the map $ev^{\C,\Gamma}_{\underline{T}}:\R H^0_{L^2}(\C;\mathcal{O})\rightarrow \R^k$
defined by \[f\mapsto ([f(T_1^1)],\dots,[f(T_1^1)\cdots f(T_{k_1}^1)],\dots,[f(T_1^m)],\dots,[f(T_1^m)\cdots f(T_{k_1}^m)])\] a \emph{local evaluation map}.
\end{defn}
\begin{prop}\label{surlocal} For any $k$-admissible graph $\Gamma\in\Theta_k$ and any point $${\underline{T}=(T_1^1,\dots,T^1_{k_1},\dots,T_1^m,\dots,T^m_{k_m})\in\Gamma},$$ the local evaluation map 
$ev^{\C,\Gamma}_{\underline{T}}$ defined in Definition \ref{local} is surjective.
\end{prop}
\begin{proof}
We remark that for any $\underline{T}\in\R^k\setminus \Delta$, the map $ev^{\C,\Gamma}_{\underline{T}}$ is equivalent to the "classical" evaluation map  $ev_{\underline{T}}:f\mapsto (f(T_1),\dots,f(T_k))$. It means that there exists an invertible matrix $A_{\underline{T}}\in GL_k(\R)$ such that $ev^{\C}_{\underline{T}}=A_{\underline{T}}\cdot ev_{\underline{T}}$. The surjectivity of  $ev^{\C,\Gamma}_{\underline{T}}$ follows from the surjectivity of $ev_{\underline{T}}$.\\
On the diagonal $\underline{x}\in\Delta\subset\R^k$, this new evaluation map $ev^{\C,\Gamma}_{\underline{T}}$ gives us more information about the higher jets of a function.
If $\underline{T}=(T_1,\dots,T_1,\dots,T_m,\dots,T_m)$, where $\#T_i=k_i$, then
$ev^{\C,\Gamma}_{\underline{T}}$ is equivalent to $$f\mapsto (f(T_1),f'(T_1),\dots,f^{(k_1-1)}(T_1),\dots,f(T_m),f'(T_m),\dots,f^{(k_m-1)}(T_m)).$$
The surjectivity of $ev^{\C,\Gamma}_{\underline{T}}$ follows.
\end{proof}

\subsubsection{Local boundedness results} The previous proposition implies that the normal Jacobian of the local evaluation map is strictly positive. We will compute the normal Jacobian of  $ev^{\C,\Gamma}_{\underline{T}}$ with respect to the metrics that we have defined, that are the $L^2$-scalar product on $\R H^0_{L^2}(\C;\mathcal{O})$ and the standard  metric  on $\R^k$.

\begin{prop}\label{boundmodel} Let $R>1$ be fixed and let $\Gamma$ be a $k$-admissible graph.  Then  there exists a positive $\epsilon_{\Gamma}>0$ such that, for every $\underline{T}\in B(0,R)\cap \Gamma$,
$$|\Jac_Nev^{\C,\Gamma}_{\underline{T}}|> \epsilon_{\Gamma}.$$
\end{prop}
\begin{proof}
 By Proposition \ref{surlocal}, for any $\underline{T}=(T_1^1,\dots,T^1_{k_1},\dots,T_1^m,\dots,T^m_{k_m})\in \Gamma$  the map $ev^{\C,\Gamma}_{\underline{T}}$ is  surjective, so its normal Jacobian is strictly positive. By compactness of $\overline{B(0,R)}$ we can find a positive $\epsilon_{\Gamma}>0$ such that $|\Jac_Nev^{\C,\Gamma}_{\underline{T}}|>\epsilon_{\Gamma}$.
\end{proof}
The following proposition shows that we can write $\Jac_Nev^{\C,\Gamma}_{\underline{T}}$ as a function of the distances $T^i_s-T^i_t$ between points lying in the same connected component and of the local  Bergman kernel $K_{\C}$ (and of its derivatives).
\begin{prop}\label{univpol} Let $\Gamma$ be a $k$-admissible graph. For any $\underline{T}\in\Gamma$ let $ev^{\C,\Gamma}_{\underline{T}}$ be the local evaluation map defined in Definition \ref{local}. 
Then the matrix associated with $(ev^{\C,\Gamma}_{\underline{T}}) (ev^{\C}_{\underline{T}})^*$ is  a symmetric $(k\times k)$ matrix composed by $m^2$ blocks, indexed by $(i,j)$ for $i,j=1,\dots,m$. The $(i,j)$-block is a $k_i\times k_j$ matrix, we denote the $(p,q)$-place of this block by $(ev^{\C,\Gamma}_{\underline{T}} ev^{\C,\Gamma *}_{\underline{T}})_{(i_p,j_q)}$. Then $(ev^{\C,\Gamma}_{\underline{T}} ev^{\C,\Gamma *}_{\underline{T}})_{(i_p,j_q)}$ is a homogenous polynomial $\mathcal{Q}_{(i_p,j_q)}$ of degree $1$ in the norm of
$K_{\C}(T^i_s,T^j_a)$ and of its derivatives, $s=1,\dots,p$, $a=1,\dots,q$. The coefficients of this polynomial are rational functions in $T^i_s-T^i_t$, $T^j_a-T^j_b$ for $1\leq t<s\leq p$, $1\leq a<b\leq q$. 
\end{prop}
\begin{proof}
We know that $\Jac_N( ev^{\C,\Gamma}_{\underline{T}})= \Jac_N (ev^{\C,\Gamma}_{\underline{x}} \mid_{(\ker ev^{\C, \Gamma}_{\underline{T}})^{\perp}})$. Let $\{f_1,\dots,f_k\}$ be an orthonormal basis of $(\ker ev^{\C, \Gamma}_{\underline{T}})^{\perp}$. 
We compute the normal Jacobian using this basis and the canonical orthonormal basis for $\R^{k_1}\times\cdots \times\R^{k_m}$.
Then the matrix of $ev^{\C,\Gamma}_{\underline{x}} \mid_{(\ker ev^{\C, \Gamma}_{\underline{T}})^{\perp}}$ associated with these orthonormal basis is a square matrix whose $i$-th column equals  the transpose of
$$
([f_i(T_1^1)],\dots,[f_i(T_1^1)\dots f_i(T_{k_1}^1)],
 \dots,[f_i(T_1^m)],\dots,[f_i(T_1^m)\dots f_i(T_{k_m}^m)]).
$$
A direct computation shows that $(ev^{\C,\Gamma}_{\underline{T}})( ev^{\C,\Gamma}_{\underline{T}})^*$ is a square  matrix with  $m^2$ blocks, indexed by $(i,j)$ for $i,j=1,\dots,m$. The $(i,j)$-block is a $k_i\times k_j$ matrix, we denote  by $(ev^{\C,\Gamma}_{\underline{T}}ev^{\C,\Gamma *}_{\underline{T}})_{(i_p,j_q)}$ the $(p,q)$-place of this block. We have that $(ev^{\C,\Gamma}_{\underline{T}}ev^{\C,\Gamma *}_{\underline{T}})_{(i_p,j_q)}$ is equal to
$$\sum_{l=1}^k[f_l(T_1^i)\dots f_l(T_p^i)][f_l(T_1^j)\dots f_l(T_q^j)].$$
By Lemma \ref{pratique}, each term $[f_l(T_1^i)\dots f_l(T_p^i)][f_l(T_1^j)\dots f_l(T_q^j)]$ is equal to 
\[\sum_{s=1}^p\sum_{r=0}^{\#T_s^i-1}c_{\underline{T}^i_p,s,r}f_l^{(r)}(T_s^i)\sum_{a=1}^q\sum_{h=0}^{\#T_a^j-1}c_{\underline{T}^j_q,a,h}f_l^{(h)}(T_a^j)\]
where $\underline{T}^i_p=(T^i_1,\dots,T^i_p)$ and $\underline{T}^j_q=(T^j_1,\dots,T^j_q)$.\\
Interchanging the sums we obtain
\[\sum_{s=1}^p\sum_{r=0}^{\#T_s^i-1}\sum_{a=1}^q\sum_{h=0}^{\#T_a^j-1}c_{\underline{T}^i_p,s,r}c_{\underline{T}^j_q,a,h}f_l^{(r)}(T_s^i)f_l^{(h)}(T_a^j).\]
Summing up to $l$ and interchanging the sums we obtain
\[\sum_{s=1}^p\sum_{r=0}^{\#T_s^i-1}\sum_{a=1}^q\sum_{h=0}^{\#T_a^j-1}c_{\underline{T}^i_p,s,r}c_{\underline{T}^j_q,a,h}{\frac{\partial^{r+h}}{\partial T^r \partial W^h}K_{\C}(T^i_s,T^j_a)}.\]
\end{proof}
\begin{oss} Remark that  the Bergman kernel $K_{\C}(Z,W)$ only depends on the distances between $Z$ and $W$. Then, Proposition \ref{univpol} implies that the matrix $(ev^{\C,\Gamma}_{\underline{T}}) (ev^{\C}_{\underline{T}})^*$ (and then $\Jac_N ev^{\C,\Gamma}_{\underline{T}}$)  only depends on the distances between the points $T^i_p$, and not on the particular position of each point.
\end{oss}
\begin{prop}\label{nuovo} Let $\Gamma$ be a $k$-admissible graph. Then, there exists a positive $\epsilon_{\Gamma}$ such that, for any $\underline{T}=(T_1^1,\dots,T^1_{k_1},\dots,T_1^m,\dots,T^m_{k_m})\in\Gamma$, we have
$$|\Jac_Nev^{\Gamma}_{\underline{T}}|>\epsilon_{\Gamma}.$$
\end{prop}
\begin{proof}
By induction on the number of connected components.\\ 
The case of one connected component is treated in Proposition \ref{boundmodel}, for $m=1$.\\
Consider now points $$\underline{T}=(T_1^1,\dots,T^1_{k_1},\dots,T_1^m,\dots,T^m_{k_m})\in\Gamma$$ where $\{T_1^i,\dots,T^i_{k_i}\}$ are the connected components of $\Gamma$, $i\in\{1,\dots,m\}$.
We consider some  connected components, say  the first $l$ ones, $\Gamma^1,\dots,\Gamma^l$ and we look at the polynomials $\mathcal{Q}_{i_p,j_a}$ defined in Proposition \ref{univpol}, $i,j\in\{1,\dots,m\}$  $p=1,\dots,k_i$, $a=1,\dots,k_j$.
For each polynomial  $\mathcal{Q}_{i_p,j_a}$, we replace all the  norms of $K_{\C}(T^i_p,T^j_a)$ and of its derivatives $i\in\{1,\dots,l\}, j\in\{ l+1,\dots,m\}$, $p\in\{1,\dots,k_i\}$, $a\in\{1,\dots,k_j\}$, by  $0$.
Geometrically, we are moving the first $l$ connected components $\Gamma^1,\dots,\Gamma^l$ to the infinity, far from the other $m-l$ components. 
After this operation, the determinant of the matrix $(ev_{\underline{T}}^{\Gamma}ev_{\underline{T}}^{\Gamma *})$ breaks into the product of the determinants of  two blocks. These  blocks represent local  evaluation maps with a smaller number of connected components, respectively $l$ and $m-l$. By induction, the determinant of  each of these two  local  evaluation maps is bounded from below by a constant which depends only on $\Gamma$.\\ 
We can apply this argument for any integer $l=1,\dots,m-1$ and any subsets of $l$ connected components of $\{\Gamma^1,\dots,\Gamma^m\}$.
By  the  continuity of the determinant,  we can then find two positive numbers  $\delta,\epsilon>0$   such that the following property holds: every time we take a point $\underline{T}\in\Gamma$ such that  the norm of  $K_{\C}(T^i_p,T^j_a)$ is  smaller than $\delta$, for some $i\neq j\in\{1,\dots,m\}$, $p=1,\dots,k_i$, $a=1,\dots,k_j$, we have $|\Jac_N ev^{\Gamma}_{\underline{T}}|>\epsilon$.
Then, we can  suppose that the norms of $K_d(T^i_p,T^j_a)$  are bigger than $\delta$, for any $i\neq j, i,j=1,\dots,m$, $p=1,\dots,k_i$, $a=1,\dots,k_j$.
In this case, there exists $R>0$ such that $\mathbf{d}(T^1_1,T^i_p)<R$ for any $i=1,\dots,m$, $p=1,\dots,k_i$ and then, by Proposition \ref{boundmodel}, we have the result.
 \end{proof}

\subsubsection{Generalized evaluation maps in normal coordinates} We now study the surjectivity of the generalized evaluation maps $ev_{\underline{x}}^{\Gamma}:\R H^0(\Sigma;\mathcal{L}^d)\rightarrow \R^{k_1}\times\cdots\times\R^{k_m}$ defined in Definition \ref{eval}.  We will use the graph notation of Section \ref{secgraph}.\\ 
We fix a $k$-admissible graph $\Gamma$, see Definition \ref{adm}. 
 Remember that this implies that  for every $$\underline{x}=(x_1^1,\dots,x^1_{k_1},\dots,x_1^m,\dots,x^m_{k_m})$$ we have $\mathbf{d}_h(x^i_p,x^i_q)\leq \frac{k}{\sqrt{d}}$ for every $i\in\{1,\dots,m\}$ and  $1\leq p,q \leq k_i$ and  that ${\mathbf{d}_h(x^i_p,x^j_q)> \frac{1}{\sqrt{d}}}$ for $i\neq j$.\\
Around the origin $x^1_1$ we consider a real normal chart $U$, see Section \ref{normal}. With a slight abuse of notation, we identify a point $x_p^i$ with its normal coordinate around $x^1_1$. We recall that  with  a real normal chart comes together with a real trivialization of $\R\mathcal{L}^d_{\mid U}$, see Section \ref{normal}. 
Under these trivializations, we consider the generalized evaluation map
\[ev^{\Gamma}_{\underline{x}}:\R H^0(\Sigma;\mathcal{L}^d)\rightarrow \R^{k_1}\times\dots\times \R^{k_m}\]
as in Definition \ref{eval}. On the space of real global sections $\R H^0(\Sigma;\mathcal{L}^d)$ we consider the $L^2$-scalar product induced by the real Hermitian metric $h$, see Section \ref{generality1}.
\subsubsection{Scaled evaluation maps} Passing to the scaled normal coordinates ${T=\sqrt{d}x}$ around $x^1_1$, we have new local coordinates $(T^i_q)_{i=1,\dots m}^{q=1,\dots k_i}$, where $T^i_q\in B(T^i_1,1)$, the ball of radius $1$ around $T^i_1$, see Section \ref{bochnerco}. 
 We then have
\[\big|[s(\frac{T_1^i}{\sqrt{d}})\dots s(\frac{T_{p}^i}{\sqrt{d}})]\big|
=\sqrt{d}^{p(p+1)/2}|[s_d(T_1^i)\dots s_d(T_{p}^i)]|\]
where $s_d(\cdot)=s(\frac{\cdot}{\sqrt{d}})$.\\
We define the map $ev^{\Gamma}_{\underline{T},d}:\R H^0(\Sigma;L^d)\rightarrow \R^{k_1}\times\dots\times \R^{k_m}$
by
\[s\mapsto \big([s_d(T_1^1)],\dots,[s_d(T_1^1)\dots s_d(T_{k_1}^1)],\dots,[s_d(T^m_1)],\dots,[s_d(T^m_1)\dots s_d(T^m_{k_m})]\big).\]
 \begin{defn}
The map $ev^{\Gamma}_{\underline{T},d}:\R H^0(\Sigma;L^d)\rightarrow \R^{k_1}\times\dots\times \R^{k_m}$ just defined is called a \emph{scaled evaluation map}. We denote its normal Jacobian by 
\[D^{\Gamma}_d=\Jac_N ev^{\Gamma}_{\underline{T},d}.\]
\end{defn}
\begin{oss}\begin{itemize}
\item The main point is that scaled evalutation maps look like local evaluation maps (see Definition \ref{local}) when $d$ tends to infinity. This fact will be proved in Propositions \ref{5}, \ref{polsigma} and \ref{onc}.
\end{itemize}
\end{oss}
\begin{prop}\label{5} Let $\Gamma$ be a $k$-admissible graph and 
 $\underline{x}$ be a point in $\Gamma_d$. We denote by  $\underline{T}\in B(x_1^1,1)^{k_1-1}\times\dots\times B(x_1^m,1)^{k_m-1}$ the scaled normal coordinates of $\underline{x}$ around its origin $x_1^1$.
Then, we have
\[\prod_{i=1}^m\frac{1}{\sqrt{d}^{\frac{k_i(k_i-1)}{2}}}\mathcal{D}^{\Gamma}_d(\underline{x})=D^{\Gamma}_d(\underline{T})\]
where  $D^{\Gamma}_d=\Jac_N ev^{\Gamma}_{\underline{T},d}.$
\end{prop}
\begin{proof}
It is a direct consequence of the change of variables $T=\sqrt{d}x$.
\end{proof}
\subsubsection{Reduction to the local model} The following result is an analogue  for the scaled evaluation maps of the Proposition \ref{univpol}.
\begin{prop}\label{polsigma}
Let $\Gamma\in\Theta_k$ be a $k$-admissible graph and $\underline{x}$ be a point in $\Gamma_d$. Let $(T^i_1,\dots,T^i_{k_1})$ be the scaled normal coordinates (around the origin $x^1_1$) of the connected component $(x^i_1,\dots,x^i_{k_i})$ of $\underline{x}$, $i\in\{1,\dots, m\}$. Consider the scaled evaluation map $$ev_{\underline{T},d}^{\Gamma}:\R H^0(\Sigma,\mathcal{L}^d)\rightarrow \R^{k_1}\times \dots \times \R^{k_m}.$$
Then the matrix of $ev^{\Gamma}_{\underline{T},d} ev^{\Gamma *}_{\underline{T},d}$ is a  symmetric $(k\times k)$ matrix composed by $m^2$ blocks, indexed by $(i,j)$ for $i,j=1,\dots,m$. The $(i,j)$-block is a $k_i\times k_j$ matrix, we denote the $(p,q)$-place of this block by $(ev^{\Gamma}_{\underline{T},d} ev^{\Gamma *}_{\underline{T},d})_{(i_p,j_q)}$. Then $(ev^{\Gamma}_{\underline{T},d} ev^{\Gamma *}_{\underline{T},d})_{(i_p,j_q)}$ is a homogenous polynomial $\mathcal{Q}_{(i_p,j_q)}$ of degree $1$ in the Bergman kernel 
$\mathcal{K}_{d}(\frac{T^i_s}{\sqrt{d}},\frac{T^j_a}{\sqrt{d}})$ and in its derivatives, $s=1,\dots,p$, $a=1,\dots,q$. The coefficients of this polynomial are rational functions in $T^i_s-T^i_t$, $T^j_a-T^j_b$,  for $1\leq t<s\leq p$, $1\leq a<b\leq q$. 
\end{prop}
\begin{proof}
It is the same proof of Proposition \ref{univpol}, replacing functions $f$ by scaled sections $s_d$.
\end{proof}
\begin{oss}\label{same} 
The homogenous polynomials $\mathcal{Q}_{(i_p,j_q)}$ of Propositions \ref{polsigma} and \ref{univpol} are the same.
\end{oss}
\begin{prop}\label{onc}
Fix $R>1$ and let $\Gamma\in\Theta_k$ be a $k$-admissible graph. 
Consider the subset of $\Gamma_d$ formed by points $\underline{x}=(x^1_1,\dots,x^1_{k_1},\dots,x^m_{1},\dots,x^m_{k_m})$ with $\mathbf{d}_h(x_p^i,x_a^j)\leq\frac{\log d}{\sqrt{d}}$ for any $i,j\in\{1,...,m\}$, $p=1,\dots,k_i$, $a=1,\dots,k_j$.\\ Let $\underline{T}=(T^1_1,\dots,T^m_{k_m})\in B(0,\log d)^{k}$ be the scaled normal coordinates of $\underline{x}$ around $x_1^1$. Then, there exists $\alpha\in (0,1)$, such that for any $d\in\mathbb{N}$ large enough and any $\underline{T}=(T^1_1,\dots,T^m_{k_m})\in B(x^1_1,\log d)^{k}$
\[\frac{1}{\sqrt{d}^k}D_d^{\Gamma}(\underline{T})=\Jac_N ev^{\C,\Gamma}_{\underline{T}}+O(\frac{1}{d^{\alpha}})\]
where $ev^{\C,\Gamma}_{\underline{T}}$ is the local evaluation map defined in Definition \ref{local}.
Moreover the error term is uniform in $T^i_p\in B(0,\log d)$ and in $x^1_1\in\R \Sigma.$
\end{prop}
\begin{proof}
This is a consequence of Proposition \ref{polsigma} and of the universality of the scaled Bergman kernel $\frac{1}{d}\mathcal{K}_d(\frac{T}{\sqrt{d}},\frac{W}{\sqrt{d}})=K_d(T,W).$
Indeed \[\frac{1}{\sqrt{d}^k}D^{\Gamma}_d(\underline{x})=\frac{1}{\sqrt{d}^k}\Jac_N ev_{\underline{T},d}^{\Gamma}=\sqrt{\frac{1}{d^k}\det (ev_{\underline{T},d}^{\Gamma} ev_{\underline{T},d}^{\Gamma *})}.\]
By multilinearity of the determinant, we can multiply each entry of the matrix $(ev_{\underline{T},d}^{\Gamma} ev_{\underline{T},d}^{\Gamma *})$ by   $\frac{1}{d}$.
We know by Proposition \ref{polsigma} that each term of this matrix $(ev_{\underline{T},d}^{\Gamma} ev_{\underline{T},d}^{\Gamma *})$ is a rational function of degree $1$ in $\mathcal{K}_d(x_p^i,x_a^j)$. This implies   that each term in the matrix $\frac{1}{d}(ev_{\underline{T},d}^{\Gamma} ev_{\underline{T},d}^{\Gamma *})$ is a rational function of degree $1$ in $\frac{1}{d}\mathcal{K}_d(x_p^i,x_a^j)=\frac{1}{\pi}K_d(T_p^i,T_a^j)$. By Corollary \ref{berglocal}, this term equals $K_{\C}(T_p^i,T_a^j)+O(\frac{1}{d^{\alpha}})$.
By  Propositions \ref{polsigma} and \ref{univpol} and  Remark \ref{same}, we have the result.
\end{proof}
\begin{proof}[\textbf{Proof of Proposition \ref{denombound}}]
It follows from Propositions \ref{onc} and \ref{nuovo}.
\end{proof} 
 
 \subsection{Estimates of the numerator $\mathcal{N}_d^{\Gamma}$}\label{snume}
This subsection is devoted to the proof  of the boundedness result for the numerator $\mathcal{N}_d^{\Gamma}$, Proposition \ref{numbound}.  We start by proving the case where the graph $\Gamma$ is connected,  see Proposition \ref{onenum}. The general case of $m$ connected components $\Gamma^1,...,\Gamma^m$ follows by an inequality of the type $\mathcal{N}_d^{\Gamma}\leq \displaystyle\prod_{i=1}^m\mathcal{N}_d^{\Gamma^i}.$\\
We use the notations of Sections \ref{olver}-\ref{sdenom}, in particular we work in normal coordinates, see Section \ref{normal}, and we use divided differences, see Definition \ref{divided}.

\begin{prop}\label{onenum} There exists $C>0$ such that for any $\underline{x}=(x_1,\dots,x_k)\in \R\Sigma^k$ with $\mathbf{d}_h(x_i,x_j)\leq \frac{k}{\sqrt{d}}$, for any $d\in\mathbb{N}$,  we have
$$\sup_{s\in S_{\underline{x}}}\frac{1}{d^k\sqrt{d}^{k(k-1)/2}}|[s(x_1)s(x_1)]|\cdot|[s(x_1)s(x_2)s(x_2)]|\cdots |[s(x_1)\dots s(x_k)s(x_k)]|<C.$$
Here, $S_{\underline{x}}$ is the unit sphere of $\R H^0_{\underline{x}}$, the kernel of the  generalized evaluation map $ev_{\underline{x}}$ defined in Definition \ref{eval}.
\end{prop}

\begin{proof}
We remember that when we use divided differences it means that we 
 consider real normal coordinates around the origin, in this case $x_1$, see Section \ref{normal}. We  identify every point $x_i$ with its coordinate, in particular we still write $x_1$ instead of $0$.\\
We pass to the scaled normal coordinates that is we consider $x_i=\frac{T_i}{\sqrt{d}}$, for $T_i\in B(x_1,k)$. We then have  
\[|[s(x_1)s(x_1)]|\cdot|[s(x_1)s(x_2)s(x_2)]|\cdots |[s(x_1)\dots s(x_k)s(x_k)]|\]
\[=\sqrt{d}^k\sqrt{d}^{k(k-1)/2}|[s_d(T_1)s_d(T_1)]|\cdot|[s_d(T_1)s_d(T_2)s_d(T_2)]|\cdots |[s_d(T_1)\dots s_d(T_k)s_d(T_k)]|\]
for any section $s$, where $s_d(\cdot)=s(\frac{\cdot}{\sqrt{d}})$. 
We still write $T_1$ instead of $0$  to  emphasize the fact that this rescaled local chart has $T_1$ as center. 
Consider the generalized $1$-jet map  $j^1_{\underline{x}}:\R H^0_{\underline{x}}\rightarrow \R^k$ 
defined by \[s \mapsto \big([s(x_1)s(x_1)],[s(x_1)s(x_2)s(x_2)],\dots,[s(x_1)\dots s(x_k)s(x_k)]\big).\]
Remark that  \[\sup_{s\in S_{\underline{x}}}\frac{1}{d^k\sqrt{d}^{k(k-1)/2}}|[s(x_1)s(x_1)]|\cdot|[s(x_1)s(x_2)s(x_2)]|\cdots |[s(x_1)\dots s(x_k)s(x_k)]|\] \[=\sup_{s\in S_{\underline{x}}\cap (\ker j^1_{\underline{x}})^{\perp}}\frac{1}{d^k\sqrt{d}^{k(k-1)/2}}|[s(x_1)s(x_1)]|\cdot|[s(x_1)s(x_2)s(x_2)]|\cdots |[s(x_1)\dots s(x_k)s(x_k)]|\]
\[=\sup_{s\in S_{\underline{x}}\cap (\ker j^1_{\underline{x}})^{\perp}}\frac{1}{\sqrt{d}^{k}}|[s_d(T_1)s_d(T_1)]|\cdot|[s_d(T_1)s_d(T_2)s_d(T_2)]|\cdots |[s_d(T_1)\dots s_d(T_k)s_d(T_k)]|.\]
 Fix an orthonormal basis
$\{\sigma^1,\dots,\sigma^k\}$ of the orthogonal of $(\ker j^1_{\underline{x}})$.\\ For any $\underline{a}=(a_1,\dots,a_k)\in S^{k-1}\subset\R^k$, write $s^{\underline{a}}=\sum_ia_i\sigma^i.$\\
\textbf{Claim 1}: \emph{for any $i,j\in\{1,\dots,k\}$, the quantity $$\frac{1}{\sqrt{d}}|[\sigma^j_d(T_1)\dots \sigma_d^j(T_i)\sigma_d^j(T_i)]|$$ is the square root of a rational function $\mathcal{Q}_j$ of degree $1$ in the norm of $K_d(T_p,T_q)$ and of its derivatives, for $p,q\in\{1,\dots,k\}$. The coefficients of $\mathcal{Q}_j$ are rational functions in the distances $T_a-T_s$ between the points, for $1\leq a\neq s \leq k$.}\\ 
Remember that $\sigma_d^j(\cdot)=\sigma^j(\frac{\cdot}{\sqrt{d}})$ and that we denoted $\frac{1}{d}\mathcal{K}_d(\frac{T_i}{\sqrt{d}},\frac{T_j}{\sqrt{d}})=K_d(T_i,T_j)$, see Section \ref{bochnerco}.
We will conclude the proof of the proposition before proving the Claim 1.
By Corollary \ref{berglocal}, $K_d(T_p,T_q)$  converges to a local universal limit. The convergence is in $C^{2k}$-topology  and the error term is an uniform $O(\frac{1}{d})$ that neither depend on the center $x_1\in \R \Sigma$ nor on $(T_2,\dots,T_k)\in B(x_1,k)^{k-1}$. In particular, by Claim 1, the function  $(T_1,\dots,T_i)\mapsto\frac{1}{\sqrt{d}}[\sigma^j_d(T_1)\dots \sigma_d^j(T_i)\sigma_d^j(T_i)]$ does not depend on $d$, up to an uniform error term $O(\frac{1}{d})$.\\
For any integer $d\in\mathbb{N}$, all points $x_1\in\R\Sigma$ and all $(T_2,\dots,T_k)\in B(x_1,k)^{k-1}$ the function $S^{k-1}\rightarrow \R$ defined by $$\underline{a}\mapsto \frac{1}{\sqrt{d}^k}|[s^{\underline{a}}_d(T_1)s^{\underline{a}}_d(T_1)]|\cdot|[s^{\underline{a}}_d(T_1)s^{\underline{a}}_d(T_2)s^{\underline{a}}_d(T_2)]|\cdots |[s^{\underline{a}}_d(T_1)\dots s^{\underline{a}}_d(T_k)s^{\underline{a}}_d(T_k)]|$$ is  continuous. Thanks to  the Claim 1, this function  does not depend on $d\in\mathbb{N}$, up to an error $O(\frac{1}{d})$ that is uniform in $x_1\in\R\Sigma$ and in $(T_2,\dots,T_k)\in B(x_1,k)^{k-1}$. 
By compactness of $\R \Sigma^k$ and of $\overline{B(x_1,k)}^{k-1}$ this function is bounded, then we have the result.\\
We will now prove the Claim 1.
The proof follows the lines of Proposition \ref{univpol}. We compute the matrix of $(j_{\underline{x}}^1)(j_{\underline{x}}^1)^*$ with respect to the basis and in scaled normal coordinates. We have 

$$j_{\underline{x}}^1=
\begin{bmatrix}
[\sigma^1_d(T_1)\sigma^1_d(T_1)] & \dots & [\sigma^k_d(T_1)\sigma^k_d(T_1)]\\
\vdots & \ddots & \vdots \\
[\sigma^1_d(T_1)\dots \sigma^1_d(T_k)\sigma^1_d(T_k)] & \dots & [\sigma^k_d(T_1)\dots \sigma^k_d(T_k)\sigma^k_d(T_k)]
\end{bmatrix}
$$
so that \[(j_{\underline{x}}^1)(j_{\underline{x}}^1)^*_{i,j}=\sum_{l=1}^k[\sigma^l_d(T_1)\dots \sigma^l_d(T_i)\sigma_l(T_i)][\sigma^l_d(T_1)\dots \sigma^l_d(T_j)\sigma^l_d(T_j)].\]
By Proposition \ref{pratique} we have that  $$[\sigma_d^1(T_1)\dots \sigma^l_d(T_j)\sigma^l_d(T_j)]=\sum_{h=1}^j\sum_{r=0}^{\#T_h-1}c_{\underline{T}_j,h,r}(\sigma^l_d)^{(r)}(T_h)$$
where $c_{\underline{x},i,r}$ is a rational function in the distances $T_s-T_t$ ,$1\leq s < t \leq k$. Here, $(\sigma^l_d)^{(r)}$ is the $r$-th  derivative with respect to the rescaled variable $T=\sqrt{d}x$.  
With the same type of computation of the Proposition \ref{univpol}, we obtain that 
\[\frac{1}{d}\sum_{l=1}^k[\sigma^l_d(T_1)\dots \sigma_d^l(T_i)\sigma^l_d(T_i)][\sigma^l_d(T_1)\dots \sigma_l(T_j)\sigma^l_d(T_j)]\]
 is a homogenous polynomial of degree $1$ in the norm of
$K_{d}(T_p,T_q)$ and of its derivatives,  $p,q\in\{1,\dots,k\}$. The coefficients of this polynomial are rational functions in $T_s-T_a$ and $e^{T_a}$ for $1\leq a\neq s \leq k$.
This implies that  $\frac{1}{\sqrt{d}}[\sigma^l_d(T_1)\dots \sigma^l_d(T_i)\sigma^l_d(T_i)]$ is  the square root of a rational function of degree $1$ in $K_d(T_l,T_m)$ and its derivatives, with coefficients that are rational functions in $T_s-T_a$  for $1\leq a\neq s \leq k$, as we have claimed.
\end{proof}
\begin{proof}[\textbf{Proof of Proposition \ref{numbound}}]
Denote by $S_{\underline{x}}$ the unit sphere in $\R H^0_{\underline{x}}$.  
Passing to polar coordinates, we have 
\[\mathcal{N}^{\Gamma}_d=C_r\int_{S_{\underline{x}}}\prod_{i=1}^{m}|[s(x^i_1)s(x^i_1)]|\cdot|[s(x^i_1)s(x^i_2)s(x^i_2)]|\cdots|[s(x^i_1)\dots s(x^i_{k_i})s(x^i_{k_i})]|ds\]
for some constant $C_r$. The measure $ds$ is induced by the $L^2$-scalar product restricted to $S_{\underline{x}}$. 
The sphere $S_{\underline{x}}$ being compact, it suffices to bound from above the integrand function.
Now, \[\sup_{s\in S_{\underline{x}}}\prod_{i=1}^{m}|[s(x^i_1)s(x^i_1)]|\cdot|[s(x^i_1)s(x^i_2)s(x^i_2)]|\cdots|[s(x^i_1)\dots s(x^i_{k_i})s(x^i_{k_i})]|\] \[\leq \prod_{i=1}^{m}\sup_{s\in S_{\underline{x}^i}}|[s(x^i_1)s(x^i_1)]|\cdot|[s(x^i_1)s(x^i_2)s(x^i_2)]\cdots|[s(x^i_1)\dots s(x^i_{k_i})s(x^i_{k_i})]| \]
where $\underline{x}^i=(x^i_1,\dots,x^i_{k_i})\in \R^{k_i}$.
We have then reduced the problem to the case of one connected component, that means the case where $\underline{x}=(x_1,\dots,x_k)$ with $\mathbf{d}_h(x_i,x_j)\leq\frac{k}{\sqrt{d}}$ and this is exactly what we have done in Proposition \ref{onenum}.
\end{proof}

\subsection{Off-diagonal estimates of the density function and proof of Proposition \ref{simple}}\label{soffest}

This subsection is devoted to the proof of Proposition \ref{simple}. We work in the same setting of  Sections \ref{ransec}, \ref{pfoo} and \ref{olvtec}.
Remember that $\mathcal{R}_d^k$ is the density function defined in Proposition \ref{denf}.
\subsubsection{About the mean and the variance} We recall two results which are the cases $k=1$ and $k=2$ of Theorem \ref{equimom}. 
 \begin{thm}(\cite[Theorem 1.1]{gw2} or \cite[Theorem 1.3]{let})\label{esper}
 For any $x\in\R\Sigma$, we have
 $$\frac{1}{\sqrt{d}}\mathcal{R}_d^1(x)=\frac{1}{\sqrt{\pi}}+O\big(\frac{1}{d}\big)$$
 in the sense of currents. Moreover the error term $O\big(\frac{1}{d}\big)$ is uniform in $x\in\R\Sigma$.
 \end{thm}
 The following is the one dimensional case of the main result of \cite[Theorem 1.6]{Puchol}.
 \begin{thm}\label{var} Under the hypothesis of Theorem \ref{equimom},
  there exists a universal  positive constant $M>0$ 
 $$\frac{1}{d}\mathbb{E}[\nu_s^2](f)=\frac{1}{\pi}\int_{(x,y)\in\R \Sigma^2}f(x,y)|\dV_h|^2+\frac{M}{\sqrt{d}}\int_{x\in \R\Sigma}f(x,x)|\dV_h|+o(\frac{1}{\sqrt{d}}).).$$
 Moreover the error term $o(\frac{1}{\sqrt{d}})$ is bounded from above by
 \[\norm{f}_{\infty}\big(O\big(\frac{1}{\sqrt{d}^{1+\alpha}}\big)+\omega_f\big(\frac{1}{\sqrt{d}^{\alpha}}\big)O(\frac{1}{\sqrt{d}})\big)\]
for any $\alpha\in(0,1)$, where  $O\big(\frac{1}{\sqrt{d}^{1+\alpha}}\big)$ and $O(\frac{1}{\sqrt{d}})$ do not depend on $f\in C^0(\R \Sigma^k)$. Here, $\omega_f(\cdot)$ is the modulus of continuity of $f$.
 \end{thm}
 The statement in \cite{Puchol} is a bit different, we explain why.
 \begin{proof}  By definition $Var(\nu_s)=\mathbb{E}[\nu_s^2]-\mathbb{E}[\nu_s]^2.$
 By Theorems \ref{esper} we have $$\frac{1}{d}Var(\nu_s)(f)=\frac{1}{d}\mathbb{E}[\nu_s^2](f)-\frac{1}{\pi}\int_{(x,y)\in\R \Sigma^2}f(x,y)|\dV_h|^2+O\big(\frac{1}{d}\big).$$
 Now, \cite[Theorem 1.6]{Puchol}, for $k=n=1$ says that $$\frac{1}{d}Var(\nu_s)(f)=\frac{M}{\sqrt{d}}\int_{x\in \R\Sigma}f(x,x)|\dV_h|+\norm{f}_{\infty}\big(O\big(\frac{1}{\sqrt{d}^{1+\alpha}}\big)+\omega_f\big(\frac{1}{\sqrt{d}^{\alpha}}\big)O(\frac{1}{\sqrt{d}})\big)$$ for a constant $M$ and for any $\alpha\in(0,1).$ 
In \cite[Theorem 1.6]{Puchol} the test function $f$ is of the form $\phi_1\phi_2$, where $\phi_i$ is a continuous function on $\R\Sigma$.    Their proof actually  works for any continuous function $f$ over $\R\Sigma^2$.
 The constant $M$ is positive (\cite[Theorem 1.8]{Puchol}) and universal, that is it does not depend on $\Sigma.$ 

 \end{proof}
\begin{oss} \begin{itemize} \item Using Olver multispace and divided differences coordinates as in Sections \ref{secgraph}-\ref{snume}, we can make the proof of \cite[Proposition 5.29]{Puchol} easier, at least for the case $k=n=1$. This proposition is fundamental for the proof of  \cite[Theorem 1.6]{Puchol}.\\
\item The universal constant in Theorem \ref{var} is not exactly the   constant appearing in \cite[Theorem 1.6]{Puchol}. This is  due to a different renormalization of the curvature form $\omega$ as well as the choice of a different Gaussian measure on $\R H^0(\Sigma;\mathcal{L}^d)$. However, they differ by a positive multiple.
\end{itemize}
 \end{oss}
As we work with the modified empirical measures $\tilde{\nu}_s^k$ (see Def. \ref{modmom}), we have to state an equivalent version of Theorem \ref{var} for the modified second moment $\mathbb{E}[\tilde{\nu}_s^2]$.
 \begin{lem}\label{var2} We follow the hypothesis of Theorem \ref{asymom} and the notations of Def. \ref{open1}. For every $f\in C^0(\R\Sigma^k)$, the following asymptotics hold:
$$\frac{1}{d}\int_{\R\Sigma^2}f\mathcal{R}^2_d|\dV_h|^2=\frac{1}{\pi}\int_{(x,y)\in\R \Sigma^2}f(x,y)|\dV_h|^2+\frac{1}{\sqrt{d}}M'\int_{x\in \R\Sigma}f(x,x)|\dV_h|+o(\frac{1}{\sqrt{d}})$$
where $M'=M-\frac{1}{\sqrt{\pi}}$ for $M$ as in Theorem \ref{var}.
Moreover the error term $o(\frac{1}{\sqrt{d}})$ is bounded from above as in Theorem \ref{var}.
 \end{lem}
 \begin{proof}
 Remember that by Proposition $\ref{denf}$, we have that  $\mathbb{E}[\tilde{\nu}_s^2](f)=\int_{\R\Sigma^2}f\mathcal{R}^2_d|\dV_h|^2.$
 By Theorem \ref{esper} and by Definition \ref{modmom},  we obtain that 
$$\frac{1}{d}\mathbb{E}[\tilde{\nu}_s^2]=\frac{1}{d}\mathbb{E}[\nu_s^2]-\frac{1}{\pi\sqrt{d}}\int_{x\in \R\Sigma}f(x,x)|\dV_h|$$
so that by Theorem \ref{var}, $$\frac{1}{d}\mathbb{E}[\tilde{\nu}_s^2](f)=\frac{1}{\pi}\int_{(x,y)\in \R \Sigma^2}f(x,y)|\dV_h|^2+\frac{1}{\sqrt{d}}M'\int_{x\in \R\Sigma}f(x,x)|\dV_h|+o(\frac{1}{\sqrt{d}}),$$
with $M'=M-\frac{1}{\sqrt{\pi}}$ and with $o(\frac{1}{\sqrt{d}})$ being as in Theorem \ref{var}.
 \end{proof}

\subsubsection{Estimates of the density function} 
 \begin{lem}\label{sim1}  Under the hypothesis of Theorem \ref{asymom}, for every $f\in C^0(\R\Sigma^k)$, the following asymptotics hold:
  $$\frac{1}{\sqrt{d}^k}\int_{U^{a,b}_d}f\mathcal{R}^k_d|\dV_h|^k=\frac{1}{\sqrt{\pi}^k}\int_{U_d^{a,b}}f|\dV_h|^k+\frac{1}{\sqrt{\pi}^{k-2}}\frac{M'}{\sqrt{d}}\int_{\R\Sigma^{k-1}}j_{ab}^*f|\dV_h|^{k-1}+o(\frac{1}{\sqrt{d}})$$
  where $M'$ is as in Lemma \ref{var2}  and the error term $o(\frac{1}{\sqrt{d}})$ is bounded from above by $$\norm{f}_{\infty}\big(O\big(\frac{1}{\sqrt{d}^{1+\alpha}}\big)+\omega_f\big(\frac{1}{\sqrt{d}^{\alpha}}\big)O(\frac{1}{\sqrt{d}})\big)$$
for any $\alpha\in(0,1)$, with $\omega_f(\cdot)$ being the modulus of continuity of $f$. 
  Moreover the error terms $O(\frac{1}{\sqrt{d}^{1+\alpha}})$ and $O(\frac{1}{\sqrt{d}})$ do not depend on $f$.
 \end{lem}
  \begin{proof} 
By Proposition \ref{broke2}, for any $\underline{x}=(x_1,\dots,x_k)\in U^{ab}_d$, we have $$\frac{1}{\sqrt{d}^k}\mathcal{R}_d^k(\underline{x})=\frac{1}{d}\mathcal{R}_d^2(x_a,x_b)\prod_{i\neq a, b}\frac{1}{\sqrt{d}}\mathcal{R}_d^1(x_i)+O(\frac{1}{d})$$
 that implies $$\int_{U^{ab}_d}f(\underline{x})\frac{1}{\sqrt{d}^k}\mathcal{R}^k_d(\underline{x})|\dV_h|^k=$$
 $$=\int_{U_d^{ab}}f(\underline{x})\frac{1}{d}\mathcal{R}_d^2(x_a,x_b)\prod_{i=1,i\neq a, b}^k\frac{1}{\sqrt{d}}\mathcal{R}_d^1(x_i)|\dV_h|^k+O(\frac{1}{d}).$$

By Theorem \ref{esper} we have that $\frac{1}{\sqrt{d}}\mathcal{R}_d^1(x_i)=\frac{1}{\pi}+O(\frac{1}{d})$ uniformly, so that
$$\int_{U_d^{ab}}f(\underline{x})\frac{1}{d}\mathcal{R}_d^2(x_a,x_b)\prod_{i=1,i\neq a, b}^k\frac{1}{\sqrt{d}}\mathcal{R}_d^1(x_i)|\dV_h|^k+O(\frac{1}{d})=$$
$$=\frac{1}{\sqrt{\pi}^{k-2}}\int_{U^{a,b}_d}f(\underline{x})\frac{1}{d}\mathcal{R}_d^2(x_a,x_b)|\dV_h|^k+O(\frac{1}{d})$$
and, by Proposition \ref{denf} and Lemma \ref{var2} this is equal to  
$$\frac{1}{\sqrt{\pi}^{k}}\int_{U_d^{a,b}}f(\underline{x})|\dV_h|^k+\frac{1}{\sqrt{\pi}^{k-2}}\frac{M'}{\sqrt{d}}\int_{\R\Sigma^{k-1}}j^*_{ab}f|\dV_h|^{k-1}+o(\frac{1}{\sqrt{d}})$$
where the error term is as in the statement of the present lemma.
\end{proof}
We are now able to prove Proposition \ref{simple}.

\begin{proof}[\textbf{Proof of Proposition \ref{simple}}]
It suffices to remark that $U_d$ is the disjoint union of the following sets (see Definition \ref{open1})  $$U_d=(\R\Sigma^k\setminus\Delta_d)\cup \bigcup_{a<b}\big(U_d^{a,b}\setminus(\R\Sigma^k\setminus \Delta_d)\big).$$
By Lemma \ref{sim2} and \ref{sim1}  we have the result. 
\end{proof}

\section{Complex case}\label{ccase}
Let $(\mathcal{L},h)$ be an ample Hermitian line bundle over a Riemann surface $\Sigma$.
The number of  zeros of a section $s\in H^0(\Sigma;\mathcal{L}^d)$ is determined by the degree of $\mathcal{L}$. However, the distribution of such zeros (in the sense of currents) depends on the chosen section.\\
In this setting Shiffman and Zelditch proved that the zero locus of a random  section $s$ of  $\mathcal{L}^d$ becomes uniformly distributed over $\Sigma$ as $d$ grows to infinity (Theorem 1.1 of \cite{sz}). They used some  estimates of the Sz\"ego kernels and the Poincar\'e-Lelong formula.
In this section we will compute  the higher moments of the zeros locus of a random section $s\in H^0(\Sigma;\mathcal{L}^d)$ in the case of $\dim\Sigma=1$.
Indeed, the methods used in this paper in the real case (Olver multispaces and Bergman kernel estimates) can be used in the same way for the complex case.

\subsection{The complex random setting} We work in the framework introduced by Shiffman and Zelditch in \cite{sz}. We restrict our definition for the dimension $1$.\\
 Let $\Sigma$ be a smooth compact Riemann surface.
Let $\mathcal{L}\rightarrow \Sigma$ be a holomorphic line bundle equipped with a Hermitian metric $h$ of positive curvature $\frac{i}{2\pi}\partial\bar{\partial}\phi=\omega\in\Omega^{(1,1)}(\Sigma,\R)$. Here, $\phi$ is the local potential of $h$. The curvature form induces a K\"ahler metric on $\Sigma$. Let $dx=\frac{\omega}{\int_{\Sigma}\omega}$ be the normalized volume form. 
\subsubsection{The Gaussian measure on $H^0(\Sigma;\mathcal{L}^d)$} The Hermitian metric $h$ induces a Hermitian metric $h^d$ on $\mathcal{L}^d$ for every integer $d>0$ and also a $L^2$-Hermitian product on the space $H^0(\Sigma;\mathcal{L}^d)$ of global holomorphic sections of $\mathcal{L}^d$ denoted by $\langle\cdot,\cdot \rangle$ and  defined by 
$$\langle\alpha,\beta\rangle=\int_{\Sigma}h^d(\alpha,\beta)dx$$
for any $\alpha,\beta$ in $H^0(\Sigma;\mathcal{L}^d)$.\\
This Hermitian product induces a Gaussian measure  on $H^0(\Sigma;\mathcal{L}^d)$
defined by
$$\mu(A)=\frac{1}{\pi^{N_d}}\int_Ae^{-\parallel s\parallel^2}ds$$
for any open subset $A\subset H^0(\Sigma;\mathcal{L}^d)$ where $ds$ is the Lebesgue measure associated with $\langle\cdot,\cdot\rangle$ and $N_d=\dim_{\C}H^0(\Sigma;\mathcal{L}^d)$.
 \subsection{Olver holomorphic multispace}
Let $X$ be a complex manifold of dimension $n$. A holomorphic curve $C$ in $X$ is an analytic one-dimensional submanifold of $X$. Such a holomorphic curve is not necessarly connected, neither closed. For example a disjoint union of open disks is an holomorphic curve.
A $(k+1)$-pointed holomorphic curve $(C;z_0,\dots,z_k)$ in $X$ is a holomorphic curve $C$ in $X$ together with $(k+1)$ not necessarly distinct points over $C$. 
Let $\mathcal{C}_h^{(k)}(X)$ be the set of all the $(k+1)$-pointed holomorphic curves in $X$ .
 The following definition is the holomorphic analogue of the Olver multispace. We call it the Olver \emph{holomorphic} multispace of a complex manifold $X$.
 \begin{defn} Two $(k+1)$-pointed holomorphic curves 
$$\textbf{C}=(C;z_0,\dots,z_k)\qquad \tilde{\textbf{C}}=(\tilde{C};\tilde{z}_0,\dots,\tilde{z}_k)$$
have $k$-th order multi-contact  if and only if 
$$z_i=\tilde{z}_i \quad\textrm{and}\quad j_{\#i-1}C\mid_{z_i}=j_{\#i-1}\tilde{C}\mid_{z_i}$$
for each $i=0,\dots,k$. The $k$-th order multi-space, denoted $X^{(k)}$ is the set of equivalence classes of $(k+1)$-pointed holomorphic curves in $X$ under the equivalence relation of $k$-th order multi-contact. The equivalence class of an $(k+1)$-pointed holomorphic curve \textbf{C} is called its $k$-th order multi-jet, and denoted $j_k\textbf{C}\in X^{(k)}$.
\end{defn}
As for the real case, local holomorphic coordinates of $X$ induces  the holomorphic divided differences coordinates for the multispace $X^{(k)}$.
\subsection{Higher moments for random holomorphic sections}
The goal of this subsection is to prove Theorem \ref{complex}.\\
We recall our setting. Let $(\mathcal{L},h)$ be a positive Hermitian line bundle over a Riemann surface $\Sigma$ and let $\omega$ be the Kahler form induced by $h$. We denote by $\dV_{\omega}^k$ the volume form on $\Sigma^k$ induced by $\omega$. For any $s\in H^0(\Sigma;\mathcal{L}^d)$, let $C_s$ be the current of integration over $\{s=0\}$, that is $C_s(f)=\sum_{x\in \{s=0\}} f(x)$. It is known that $\frac{1}{d}\mathbb{E}[C_s]= \omega+O(\frac{1}{d})$ in the sense of currents and actually that $\frac{1}{d}C_s\rightarrow \omega$ almost surely (see \cite{sz}). 
 \begin{proof}[\textbf{Proof of Theorem \ref{complex}}]
The proof follows the lines of the proof of Theorem \ref{equimom}, so we only sketch the proof.
By coarea formula, we can write $$\mathbb{E}[C_s^k](f)=\int_{\Sigma^k}f \mathcal{R}_d^k \dV_{\omega}^k$$
Define $U_d=\cup_{a< b} U_d^{a,b}$ where $U_d^{a,b}$ is the  set
$$\displaystyle U_d^{a,b}\doteqdot \{(x_1,\dots\dots,x_k)\in \Sigma^k \mid  \mathbf{d}_h(x_i,x_j)> \frac{\log d}{C'\sqrt{d}}\hspace{0.5mm} \Rightarrow \hspace{0.5mm} (i,j)=(a,b)\}$$
where $C'$ is the constant appearing in Theorem \ref{fardiag}.
On $U_d$ we have
$$\frac{1}{d^k}\int_{U_d}f \mathcal{R}_d^k \omega_k=\int_{U_d}f\dV_{\omega}^k+\norm{f}_{\infty}O(\frac{1}{d})$$
 where $O(\frac{1}{d})$ does not depend on $f$.
This is as in Proposition \ref{simple}.  The role of Theorems \ref{esper} and \ref{var} here is played respectively by \cite[Proposition 3.2]{sz} and \cite[Theorem 1.1]{sz2}.
By using Olver multispaces and divided differences coordinates as in Sections \ref{42}-\ref{snume}, we  prove that $\frac{1}{d^k}\mathcal{R}_d^k $ is uniformly bounded, so that $$\frac{1}{d^k}\int_{\Sigma^k\setminus  U_d}f \mathcal{R}_d^k\dV_{\omega}^k=\norm{f}_{\infty}O(\frac{(\log d)^4}{d^{2}}).$$
Putting together the integral over $U_d$ and over $\Sigma^k\setminus  U_d$ we obtain the result.
\end{proof}
\section*{Acknowledgements}{I would like to thank my advisor Jean-Yves Welschinger for all the fruitful discussions we had. I would also like to thank Thomas Letendre for his remarks and for his careful reading of the text.   This work was performed within the framework of the LABEX MILYON (ANR-10-LABX-0070)
of Universit\'e de Lyon, within the program "Investissements d'Avenir"
(ANR-11-IDEX-0007) operated by the French National Research Agency (ANR). }

\bibliographystyle{plain}
\bibliography{biblio}

\end{document}